\documentclass{amsart}
\usepackage{bbm,amsthm,amsmath,amsfonts,amscd,amssymb}
\usepackage{stmaryrd}
\usepackage[all]{xy}
\usepackage{enumerate}
\usepackage{color}
\usepackage{hyperref}
\theoremstyle{plain}

\newcommand{\printname}[1] {}

\newtheorem{theorem}{Theorem}[section]
\newtheorem{proposition}[theorem]{Proposition}
\newtheorem{lemma}[theorem]{Lemma}
\newtheorem{corollary}[theorem]{Corollary}

\newtheorem{definition}[theorem]{Definition}
\newtheorem{example}[theorem]{Example}
\newtheorem{remark}[theorem]{Remark}

\numberwithin{equation}{section}

\newcommand{\R}{\mathbb R}
\newcommand{\Z}{\mathbb Z}

\newcommand{\rmap}{\longrightarrow}

\newcommand*{\longhookrightarrow}{\ensuremath{\lhook\joinrel\relbar\joinrel\rightarrow}}

\DeclareMathOperator*{\add}{+}  
\newcommand{\soma}[1] {\displaystyle \add_{\scriptscriptstyle #1} \,}    

\DeclareMathOperator*{\sub}{-}  
\newcommand{\dif}[1] {\displaystyle \sub_{\scriptscriptstyle #1} \,}

\newcommand{\leftexp}[2]{{\vphantom{#2}}^{\scriptscriptstyle{#1}}{#2}}   
\newcommand{\rightexp}[2]{#1^{\scriptscriptstyle{#2}}}   

\newcommand{\bbd}{\mathbb{D}}           
\newcommand{\dec}[1] {\mathrm{Dec}(#1)} 
\newcommand{\hor}{\rm hor}
\newcommand{\ver}{\rm ver}
\newcommand{\id}{\mathrm{id}}
 
\newcommand{\Lie}{\mathcal{L}}          
\newcommand{\X}{\mathcal{X}}            
\newcommand{\Y}{\mathcal{Y}}            

\newcommand{\Arrow}{\rightarrow}        
\newcommand{\lrrow}{\longrightarrow}    
\newcommand{\hrrow}{\hookrightarrow}    

\newcommand{\D}{\mathcal{D}}                              
\newcommand{\Hom}{\operatorname{Hom}}                     
\newcommand{\End}{\operatorname{End}}                     
\newcommand{\VB}{\mathcal{VB}}                            
\newcommand{\hnabla}{\nabla^{\scriptscriptstyle \Hom}}    
\newcommand{\uHom}{\underline{\Hom}}                      
\newcommand{\uEnd}{\underline{\End}}                      
\newcommand{\Rep}{\R \mathrm{ep}}                         

\newcommand{\comment}[1]{}

\begin{document}
\title{$\VB$-algebroid morphisms and representations up to homotopy}

\author{T. Drummond }
\address{Universidade Federal do Rio de Janeiro, Instituto de Matem\'atica, 21945-970, Rio de Janeiro - Brazil.}
\email{drummond@impa.br}

\author{M. Jotz Lean}
\address{School of Mathematics and Statistics, The University of Sheffield, Hicks Building, Hounsfield Road, Sheffield, S37RH, United Kingdom.} 
\email{M.Jotz-Lean@sheffield.ac.uk}


\author{C. Ortiz}
\address{Instituto de Matemática e Estatística, Universidade de São Paulo, Rua do Matão 1010, Cidade Universitária, 05508-090, São Paulo, Brazil} 
\email{cortiz@ime.usp.br}



\begin{abstract}

  We show in this paper that the correspondence between $2$-term
  representations up to homotopy and $\mathcal{VB}$-algebroids,
  established in \cite{GrMe10a}, holds also at the level of morphisms.
  This correspondence is hence an equivalence of
  categories.  As an application, we study foliations and
  distributions on a Lie algebroid, that are compatible both with the linear
  structure and the Lie algebroid structure. In particular, we show
  how infinitesimal ideal systems in a Lie algebroid $A$ are
  related with subrepresentations of the adjoint representation of
  $A$.

\end{abstract}
\maketitle
\tableofcontents

\setcounter{tocdepth}{0}

\section{Introduction}             
There are several definitions of ideals in Lie algebroids.  The most
obvious one is the following: Let $(q\colon A\to M, \rho,
[\cdot\,,\cdot])$ be a Lie algebroid.  An ideal in $A$ is a subbundle
$I\subseteq A$ over M, such that the space of sections $\Gamma(I)$ is
an ideal in $\Gamma(A)$ endowed with the Lie bracket
$[\cdot,\cdot]$. The first immediate consequence of this definition is
the inclusion $I\subseteq \ker(\rho)$, which shows that $I$ is totally
intransitive. This notion of ideal is hence not very useful.  For
instance, an ideal in this sense only corresponds to a surjective
morphism of algebroids over the same base.  Mackenzie defines
\emph{ideal systems} in his book \cite{Mackenzie05} and shows that the
kernel of a fibration of Lie algebroids is an ideal system in his
sense.  Two of the authors define in \cite{JoOr13} an infinitesimal
version of this.  Infinitesimal ideal systems appear naturally in the
study of multiplicative foliations on Lie groupoids, a subject which
have drawn some attention in connection to geometric quantization of
Poisson manifolds \cite{Hawkins08} and also in a modern approach to
Cartan's work on Lie pseudogroups \cite{CrSaSt12}.  Multiplicative
foliations on a Lie group are in one-to-one correspondence with ideals
in its Lie algebra \cite{Ortiz08,Jotz11}.  An ideal in a Lie algebra
is a subrepresentation of its adjoint representation.  Hence, with the proper
notion of adjoint representation of Lie algebroids, one expects
infinitesimal ideal systems to be equivalent to some
subrepresentations of the adjoint representation. This paper explains
how infinitesimal ideal systems and the adjoint representation of Lie
algebroid defined by Gracia-Saz and Mehta \cite{GrMe10a} and
independently by Arias Abad and Crainic \cite{ArCr12} are related.

\medskip

The approach of Gracia-Saz and Mehta to study Lie algebroid
representations is to view them as $\VB$-algebroids. For instance, flat
$A$-connections on a vector bundle $E \Arrow M$ are in one-to-one
correspondence with $\VB$-algebroid structures on
$$
\begin{CD}
A\oplus E @>>> E\\
@VVV   @VVV\\
A @>>> M.
\end{CD}
$$
In general, flat $A$-superconnections correspond to splittings of a
canonical short exact sequence of vector bundles, that is associated
to the given $\VB$-algebroid.
The tangent prolongation of a Lie
algebroid is for instance a $\VB$-algebroid that corresponds to the adjoint
representation by splitting (via the choice of a connection on $A$)
the following exact sequence
$$
0 \longrightarrow T^*M\otimes A \longrightarrow J^1(A) \longrightarrow
A \longrightarrow 0,
$$
where $J^1A$ is the first order jet bundle of $A$.
Note that flat superconnections are the two-term case of the representations
up to homotopy defined by Arias Abad and Crainic in
\cite{ArCr12}.

Our main result on $\VB$-algebroids is
the existence of a  one-to-one correspondence between morphisms
of representation up to homotopy on 2-term complexes and morphisms of
$\VB$-algebroids (Theorem \ref{main}).
This shows that the
correspondence established by Gracia-Saz and Mehta in
\cite{GrMe10a} is actually the correspondence of objects in an equivalence of categories between the
category of representation up to homotopy and the category of
$\VB$-algebroids. 

We then apply our results on morphisms to the inclusion of
distributions inside the tangent bundle (see Theorem \ref{ideals}).
We obtain the equivalence of infinitesimal ideal systems in a Lie algebroid with
subrepresentations of its adjoint  and double
representations up to homotopy. We also discuss the case of general
(non-integrable) subbundles of the tangent of a Lie
algebroid. In that case, the representations up to homotopy lead to a
new interpretation of the infinitesimal description of
multiplicative distributions obtained by \cite{CrSaSt12} (see Theorem
\ref{IM_prop}).

\medskip

This paper is organized as follows. Sections 2 and 3 recall
background knowledge on representations up to homotopy and double
vector bundles.  Section 4 establishes the one-to-one
correspondence between $\VB$-algebroid morphisms and morphisms of
representations up to homotopy on 2-term complexes. We revisit Lie
bialgebroids and IM-2 forms from the perspective of morphisms of
representations up to homotopy. Section 5 studies (integrable and
non integrable) subbundles of the tangent of a Lie algebroid, that are
compatible with the linear and with the Lie algebroid structure.

We show in the appendix that the dictionary between
$\mathcal{VB}$-algebroids and $2$-term representations up to homotopy
is compatible with dualizations
in each category.\\

\textbf{Acknowledgements:} The authors would like to thank Henrique
Bursztyn and Jim Stasheff for useful comments that have improved the presentation of
this work.  \comment{Thiago: I tried to talk about this with Camilo
  but he never gave any valuable suggestion.}  \comment{Madeleine:
  yes, I know what you mean... What about Jim Stasheff, Cristian? Do
  you want to thank him?}  Drummond acknowledges support of
CAPES-FCT at IST-Lisboa, where part of this work was developed.  Jotz
was supported by the Dorothea-Schl\"ozer program of the University of
G\"ottingen, and a fellowship for prospective researchers of the Swiss
NSF (PBELP2\_137534) for research conducted at UC Berkeley, the
hospitality of which she is thankful for.  Ortiz would like to thank
IMPA (Rio de Janeiro) for a 2012-Summer Postdoctoral Fellowship and
its hospitality while part of this work was carried out.


\section{Representations up to homotopy of Lie algebroids}

We recall here some background material on representations up to
homotopy. We  follow mostly \cite{ArCr12}.

\subsection{Definition and examples.}

Let $E \Arrow M$ be a vector bundle and $V= \bigoplus_{k\in
  \mathbb{Z}} V_k$  a graded vector bundle. The space of $V$-valued
$E$-differential forms, $\Omega(E; V):=\Gamma(\wedge^{\bullet}
E^*\otimes V)$, has a grading given by
$$
\Omega(E;V)_k = \bigoplus_{i+j=k}\Gamma(\wedge^i E^* \otimes V_j)
$$
and a natural  (graded-)module structure over the algebra $\Omega(E):= \Gamma(\wedge^{\bullet}E^*)$.

If $W = \bigoplus_{k\in \mathbb{Z}} W_k$ is a second graded vector
bundle, $\underline{\Hom}(V, W)$ is the graded vector bundle whose
degree $k$ part is
$$
\underline{\Hom}(V, W)_k = \bigoplus_{i\in \mathbb{Z}} \Hom(V_i, W_{i+k}).
$$

\medskip

Let now $(A, [\cdot, \cdot], \rho_A)$ be a Lie algebroid over $M$.
\begin{definition}
  A \textit{homogeneous} $A$-connection $\nabla$ on the graded vector
  bundle $V=\bigoplus_{k\in \Z} V_k$ is an $A$-connection $\nabla:
  \Gamma(A) \times \Gamma(V) \Arrow \Gamma(V)$ such that $\nabla_a$
  preserves $\Gamma(V_k)$, for all $k \in \mathbb{Z}$ and $a \in
  \Gamma(A)$. Equivalently, an $A$-connection on $V$ is given by a
  family $\{\nabla^k\}_{k\in \mathbb{Z}}$, where each $\nabla^k$ is an
  $A$-connection on $V_k$.
\end{definition} 

From now on, we assume that all connections on graded vector bundles are homogeneous.

\begin{definition}
  Let $V$ be a graded vector bundle. A representation up to homotopy
  of $A$ on $V$ is a degree one map $ \D:
  \Omega(A;V)_{\scriptscriptstyle \bullet} \Arrow
  \Omega(A;V)_{\scriptscriptstyle \bullet + 1} $ such that $\D^2=0$
  and
\begin{equation}\label{D_eq}
  \D(\alpha \wedge \omega) = d_A\alpha \wedge \omega + (-1)^{k} \alpha \wedge \D(\omega), \text{ for } \alpha \in \Omega^k(A), \, \omega \in \Omega(A;V),
\end{equation}
where $d_A: \Omega^{\bullet}(A) \Arrow \Omega^{\bullet+1}(A)$ is the Lie algebroid differential
\begin{align*}
d_A\alpha(a_1, \dots, a_{k+1}) & = \sum_{i=1}^k (-1)^{i+1} \Lie_{\rho_A(a_i)}\alpha(a_1, \dots, a_{i-1}, a_{i+1},  \dots, a_k)\\
& \hspace{-20pt}+ \sum_{1\leq i < j \leq k} (-1)^{i+j}\alpha([a_i,a_j], a_1, \dots, a_{i-1}, a_{i+1}, \dots, a_{j-1}, a_{j+1}, \dots, a_k).
\end{align*}
\end{definition}

\begin{definition}\label{rep_morf}
  \em A morphism between two representations up to homotopy of $A$ is
  a degree zero $\Omega(A)$-linear map
$$
\Omega(A; V) \Arrow \Omega(A; W)
$$
which intertwines the differentials $\D_W$ and $\D_V$. We denote it by
$(A, V) \Rightarrow(A, W)$.
\end{definition}

In this paper we are mostly concerned with
representations up to homotopy on graded vector bundles $V$ concentrated in degree 0
and 1. These are called \textit{2-term graded vector bundles} and the
representations up to homotopy of $A$ on 2-term graded vector bundles form a category which we denote by $\R
\mathrm{ep}^{2}(A)$. We denote by $\mathrm{Rep}^2(A)$ the set of
isomorphism classes of objects of $\R \mathrm{ep}^{2}(A)$.

For a 2-term representation up to homotopy $V \in \R
\mathrm{ep}^2(A)$, the derivation property \eqref{D_eq} implies that the differential $\D: \Omega(A; V) \Arrow
\Omega(A;V)$ is determined by
\begin{enumerate}
\item [(1)] a bundle map $\partial: V_0 \to V_1$;
\item [(2)] an $A$-connection $\nabla$ on $V$ compatible with
  $\partial$ (i.e.~$\partial \circ \nabla^0 = \nabla^1
  \circ \partial$);
\item [(3)] an element $K \in \Omega^2(A, \uEnd(V)_{-1})= \Omega^2(A,
  \Hom(V_1, V_0))$ such that $d_{\nabla^{\scriptscriptstyle \End}}K=0$
  and the diagram below commutes
$$
\xy
(0,0)*+{V_0}="A"; (20,0)*+{V_1}="B"; (0,20)*+{V_0}="C"; (20,20)*+{V_1}="D";
{\ar@{->}_{\partial}"A";"B"};
{\ar@{->}^{\partial}"C";"D"};
{\ar@{->}_{R_{\nabla^0}} "C"; "A"}; 
{\ar@{->}^{R_{\nabla^1}} "D"; "B"};
{\ar@{->}_{-K} "D"; "A"}
\endxy
$$
where $R_{\nabla^i}$ is the curvature of $\nabla^i$, for $i=0,1$. We
say that $(\partial, \nabla, K)$ are \textit{the structure operators}
for $V \in \R \mathrm{ep}^{2}(A)$.
\end{enumerate}
We refer to \cite{ArCr12} for a detailed exposition of the correspondence $\D \mapsto (\partial, \nabla, K)$ (pointing out that our sign convention for $K$ is different from the one in \cite{ArCr12}).

For $V, W \in \Rep^2(A)$, a morphism $(A, V) \Rightarrow (A, W)$ is determined
by a triple $(\phi_0, \phi_1, \Phi)$, where $\phi_0: V_0 \Arrow W_0$,
$\phi_1: V_1 \Arrow W_1$ are bundle maps and $\Phi \in \Omega^1(A;
\Hom(V_1, W_0))$, satisfying 
\begin{equation}\label{comp1}
\phi_1\circ \partial_V =  \partial_W \circ\phi_0,\\
\end{equation}
\begin{equation}\label{comp2}
\hnabla_a (\phi_0, \phi_1) = (\Phi_{a} \circ \partial_V, \,\partial_W
\circ \Phi_{a}) \quad \text{ for all } a\in\Gamma(A)\\
\end{equation}
and
\begin{equation}\label{comp3}
d_{\hnabla}\Phi= \phi_0 \circ K_V - K_W\circ \phi_1 ,
\end{equation}
where $\hnabla$ is the $A$-connection on $\uHom(V,W)$ (see
\cite{ArCr12} for more details).

\medskip

In the following, given vector bundles $E, E'$ over $M$, we denote by
$E_{[0]}\oplus E'_{[1]}$ the graded vector bundle consisting of $E$ in degree
$0$ and of $E'$ in degree $1$.

\begin{example}[Double representation]\label{double}\em
  Let $B \Arrow M$ be a vector bundle and consider the graded vector
  bundle $V=B_{[0]}\oplus B_{[1]}$.  Any connection $\nabla:
  \Gamma(TM)\times \Gamma(B) \Arrow \Gamma(B)$ induces a
  representation up to homotopy of $TM$ on $V$ by taking
  $\partial=\mathrm{id}_B$, $\nabla^0=\nabla^1 = \nabla$ and $K =
  -R_{\nabla}$, the curvature of $\nabla$. The isomorphism class of
  this representation does not depend on the choice of $\nabla$ and is
  called the \textit{double representation of $TM$ on $B$}. We denote
  it by $\mathcal{D}(B) \in \mathrm{Rep}^{2}(TM)$ and the
  representation itself by $\mathcal{D}_{\nabla}(B) \in \R
  \mathrm{ep}^2(A)$.
\end{example}

\begin{example}[Adjoint representation]\label{adjoint}
  Let $(A, [\cdot,\cdot]_A, \rho_A)$ be a Lie algebroid over $M$.  Any
  connection $\nabla: \Gamma(TM) \times \Gamma(A) \Arrow \Gamma(A)$ on
  $A$ induces a representation up to homotopy of $A$ on
  $V=A_{[0]}\oplus TM_{[1]}$ in the following manner. The map
  $\partial$ is just the anchor $\rho_A: A \Arrow TM$. The
  $A$-connection $\nabla^{\rm bas}$ on $V$ (called the \textit{basic
    connection}) has degree zero and degree one parts given by
  \hspace{-35pt}
$$
\begin{array}{rccl}
\nabla^{\rm bas}:&  \Gamma(A) \times \Gamma(A)& \longrightarrow & \Gamma(A)\\
             &    (a,b) & \longmapsto & [a, b]_A + \nabla_{\rho_A(b)}a. 
\end{array}
$$
and 
$$
\begin{array}{rccl}
\nabla^{\rm bas}:&  \Gamma(A) \times \Gamma(TM)& \longrightarrow & \Gamma(TM)\\
             &    (a,X) & \longmapsto & [\rho_A(a), X]_A + \rho_A(\nabla_X a) \vspace{5pt},
\end{array}
$$
respectively. The element $K$ is the \textit{basic curvature}
$R_{\nabla}^{\rm bas} \in \Omega^2(A, \operatorname{Hom}(TM, A))$
defined by
$$
R_{\nabla}^{\rm bas}(a, b)(X) = \nabla_X [a, b] - [\nabla_X a, b] -[a, \nabla_X b]  + \nabla_{\nabla_a^{\rm bas} X } \,b 
- \nabla_{\nabla_b^{\rm bas} X }\, a.
$$
As before, the isomorphism class of this representation does not
depend on the choice of $\nabla$ and it is called the \textit{adjoint
  representation of $A$}. We denote it by $\mathrm{ad} \in
\mathrm{Rep}^{2}(A)$ and the representation itself by
$\mathrm{ad}_{\nabla} \in \R \mathrm{ep}^2(A)$.
\end{example}

Given a 2-term representation $V \in \R \mathrm{ep}^2(A)$ with
structure operators $(\partial, \nabla, K)$ of $A$ on $V= V_0 \oplus
V_1$, its dual is the representation $V^{\top} \in \R
\mathrm{ep}^2(A)$, where $V^{\top}_0 = V_1^*$, $V^{\top}_1 = V_0^*$,
with structure operators given by
\begin{equation}\label{dual_struct}
\partial_{V^{\top}}= \partial^*, \,\,\, \nabla^{V^{\top}} = \nabla^* \text{ and } K_{V^{\top}} = - K^*
\end{equation}
where $\nabla^*$ is the $A$-connection dual to $\nabla$, given by
\begin{equation}\label{dual_conn}
  \langle \nabla^*_a \xi, v \rangle + \langle \xi, \nabla_a v \rangle = \Lie_{\rho_A(a)} \langle \xi, v \rangle, \,\, \forall v \in \Gamma(V),  \, \xi \in \Gamma(V^{\top}).
\end{equation}

\begin{example}[Coadjoint representation]
  The coadjoint representation is the representation of $A$ on
  $T^*M_{[0]}\oplus A^*_{[1]}$ dual to the adjoint representation. It
  is denoted by $\mathrm{ad}^{\top}(A) \in \mathrm{Rep}^{2}(A)$.
\end{example}

\subsection{Pullbacks.}
We define here morphisms between 2-term representations up to homotopy of
different Lie algebroids. Let $(A', [\cdot, \cdot]_{A'}, \rho_{A'})$
be another Lie algebroid over $M$ and $T: A \Arrow A'$ a Lie
algebroid morphism over $\mathrm{id}_M$. Choose  $W \in \R
\mathrm{ep}^2(A')$ with structure operators $(\partial, \nabla, K)$.

Define $\nabla^{T}: \Gamma(A) \times \Gamma(W) \Arrow \Gamma(W)$ to be
the $A$-connection given by
\begin{equation}\label{pull_con}
\nabla^{\scriptscriptstyle T}_a w := \nabla_{T(a)} w
\end{equation}
for $a \in \Gamma(A)$ and $w \in \Gamma(W)$ and $T^*K \in \Omega^2(A; \operatorname{Hom}(W_1, W_0))$ by
\begin{equation}
T^*K (a_1, a_2) = K(T(a_1), T(a_2)), \,\, (a_1, a_2) \in A\times_M A.
\end{equation}

\begin{lemma}
 The triple $(\partial, \nabla^T, T^*K)$ defines structure operators for a representation up to homotopy of $A$ on $W$ which is called the pullback of $W$ by $T$ and it is denoted by $T^{!}W \in \Rep^2(A)$.
\end{lemma}

\begin{proof}
We leave the details to the reader.
\end{proof}

\begin{example}
If $T$ is the inclusion of a Lie subalgebroid $A \hookrightarrow
A'$, the pullback $T^{\,!} W$ is just the restriction of the
representation to $A$.
\end{example}

The usefulness of taking pullbacks is that it allows one to define
morphisms between representations up to homotopy of different
algebroids.

\begin{definition}
  Let $W \in \R \mathrm{ep}^{2}(A')$ and $V \in \R
  \mathrm{ep}^{2}(A)$ be representations up to homotopy. We
  define a morphism $(A, V) \Rightarrow (A', W)$ over a Lie algebroid
  morphism $T: A \Arrow A'$ to be an usual morphism $(A, V)
  \Rightarrow (A, T^{\,!}W)$ (as given in Definition \ref{rep_morf}).
\end{definition}

\begin{remark}\em
The pullback operation can be defined for arbitrary representations up to homotopy. It was already defined in this generality for representations up   to homotopy of Lie groupoids in \cite{ArCr12}. Also, the pullback
can be extended to morphisms and we get a functor $T^{\,!}: \R  \mathrm{ep}^{2}(A') \Arrow \R \mathrm{ep}^{2}(A)$.
\end{remark}


\section{Double vector bundles.}

We briefly recall the definitions of double vector bundles,
of some of their special sections and of their morphisms. We refer to \cite{Mackenzie05} for a more
detailed treatment (see also \cite{GrMe10a} for a treatment closer to
ours). We also classify subbundles of double
vector bundles.

\subsection{Preliminaries.}
\begin{definition}
A double vector bundle is a commutative square
$$
\begin{CD}
D @>q_B^D>> B\\
@Vq_A^DVV   @VVq_BV\\
A @>>q_A> M
\end{CD}
$$
satisfying the following three conditions:
\begin{itemize}
\item[DV1.] all four sides are vector bundles;
\item[DV2.] $q_B^D$ is a vector bundle morphism over $q_A$;
\item[DV3.] $\soma{B}: D\times_B D \rightarrow D$ is a vector bundle
  morphism over $+: A\times_M A \rightarrow A$, where $\soma{B}$ is
  the addition map for the vector bundle $D\rightarrow B$.
\end{itemize}
\end{definition}

Given a double vector bundle $(D; A, B; M)$, the vector bundles $A$ and $B$ are called
the \textit{side bundles}. The zero sections are denoted by
$\rightexp{0}{A}: M \rightarrow A$, $\rightexp{0}{B}: M \rightarrow
B$, $\leftexp{A}{0}: A \rightarrow D$ and $\leftexp{B}{0}: B
\rightarrow D$. Elements of $D$ are written $(d; a, b; m)$,
where $d \in D$, $m\in M$ and $a=q_A^D(d) \in A_m$, $b=q_B^D(d)\in
B_m$.

The \textit{core} $C$ of a double vector bundle is the intersection of the kernels of
$q_A^D$ and $q_B^D$. It has a natural vector bundle structure over
$M$, the projection of which we call $q_C: C \rightarrow M$. The
inclusion $C \hookrightarrow D$ is usually denoted by
$$
C_m \ni c \longmapsto \overline{c} \in (q_A^D)^{-1}(\rightexp{0}{A}_m) \cap (q_B^{D})^{-1}(\rightexp{0}{B}_m).
$$

\begin{definition}
  Let $(D; A, B; M)$ and $(D'; A', B'; M)$ be two double vector
  bundles. A double vector bundle morphism $(F; F_{\ver}, F_{\hor}; f)$ from $D$ to
  $D'$ is a commutative cube
$$
\xy
(0,0)*+^{A}="A"; (18,8)*+^{M}="M_1"; (30,0)*+^{A'}="A'";(48,8)*+^{M}="M_2";
(0,20)*+^{D}="D"; (18,28)*+^{B}="B_1"; (30,20)*+^{D'}="D'";(48,28)*+^{B'}="B'";
{\ar@{->}"A"; "M_1"}; 
{\ar@{->}_{F_{\ver}}"A"; "A'"};
{\ar@{->}^{f}"M_1"; "M_2"};
{\ar@{->}"A'"; "M_2"};
{\ar@{->}"D"; "B_1"}; 
{\ar@{->}^{F}"D";"D'"};
{\ar@{->}^{F_{\hor}}"B_1"; "B'"};
{\ar@{->}"D'"; "B'" };
{\ar@{->}"D"; "A"}; 
{\ar@{->}"B_1";"M_1"};
{\ar@{->}"D'"; "A'"};
{\ar@{->}"B'"; "M_2"};
\endxy
$$
where all the faces are vector bundle morphisms. 
\end{definition}

Given a double vector bundle morphism $(F; F_{\ver}, F_{\hor}; f)$, its restriction to
the core bundles induces a vector bundle morphism $F_{c}: C \Arrow
C'$. In the following, we are mainly interested in double vector bundle morphisms where
$f=\mathrm{id}_M: M \Arrow M$. In this case, we omit the reference to
$f$ and denote a double vector bundle morphism by $(F; F_{\ver}, F_{\hor})$.

\medskip

Given a double vector bundle $(D;A,B;M)$, the space of sections $\Gamma(B, D)$ is
generated as a $C^{\infty}(B)$-module by two distinguished classes of
sections (see \cite{Mackenzie11}), the \textit{linear} and the
\textit{core sections} which we now describe.

\begin{definition}
  For a section $c: M \rightarrow C$, the corresponding \textit{core
    section} $\hat{c}: B \rightarrow D$ is defined as
\begin{equation}\label{core_section}
\hat{c}(b_m) = \leftexp{B}{0}_{\vphantom{1}_{b_m}} \soma{A} \overline{c(m)}, \,\, m \in M, \, b_m \in B_m.
\end{equation}
\end{definition}

We denote the space of core sections by $\Gamma_c(B, D)$.

\begin{definition}
  A section $\X \in \Gamma(B, D)$ is called \textit{linear} if $\X: B
  \rightarrow D$ is a bundle morphism from $B \rightarrow M$ to $D
  \rightarrow A$. The space of linear sections is denoted by
  $\Gamma_{\ell}(B, D)$.
\end{definition} 

The space of linear sections is a locally free $C^{\infty}(M)$-module
(see e.g.~\cite{GrMe10a}). Hence, there is a vector bundle $\widehat{A}$
over $M$ such that $\Gamma_{\ell}(B, D)$ is isomorphic to
$\Gamma(\widehat{A})$ as $C^{\infty}(M)$-modules. Note that for a
linear section $\X$, there exists a section $\X_0: M \rightarrow A$
such that $q_A^D \circ \X = \X_0 \circ q_B$. The map $\X \mapsto \X_0$
induces a short exact sequence of vector bundles
\begin{equation}\label{fat_seq}
0 \longrightarrow B^*\otimes C \longhookrightarrow \widehat{A} \longrightarrow A \longrightarrow 0,
\end{equation}
where for $T \in \Gamma(B^*\otimes C)$, the corresponding section
$\widehat{T} \in \Gamma_{\ell}(B, D)$ is given by
\begin{equation}\label{core_morf}
\widehat{T}(b_m) = \leftexp{B}{0}_{b_m} \soma{A} \overline{T(b_m)}.
\end{equation}
We call splittings $h: A \rightarrow \widehat{A}$ of the short exact
sequence \eqref{fat_seq} \textit{horizontal lifts}.

\begin{example}\label{trivial_dvb}
  Let $A, \, B, \, C$ be vector bundles over $M$ and consider
  $D=A\oplus B \oplus C$. With the vector bundle structures
  $D=q^{!}_A(B\oplus C) \Arrow A$ and $D=q_B^{!}(A\oplus C) \Arrow B$,
  one has that $(D; A, B; M)$ is a double vector bundle called the
  \textit{trivial double vector bundle
    with core $C$}. The core sections are given by
$$
b_m \mapsto (\rightexp{0}{A}_m, b_m, c(m)), \text{ where } m \in M, \, b_m \in B_m, \, c \in \Gamma(C).
$$
The space of linear sections $\Gamma_{\ell}(B, D)$ is naturally
identified with $\Gamma(A)\oplus \Gamma(B^*\otimes C)$ via
$$
(a, T): b_m \mapsto (a(m), b_m, T(b_m)), \text{ where } T \in \Gamma(B^*\otimes C), \, a\in \Gamma(A).
$$
The canonical inclusion $\Gamma(A)
\hookrightarrow \Gamma_{\ell}(B,D)$ is a horizontal lift.

Let $A', B', C'$ be another triple of vector bundles over $M$ and
consider the corresponding trivial double vector bundle with core $C'$, $D'=A'\oplus
B'\oplus C'$. Any double vector bundle morphism $(F; F_{\ver}, F_{\hor})$ from $D$ to
$D'$ is given by
\begin{equation}\label{trivial_morf}
(a, b, c) \mapsto (F_{\ver}(a), F_{\hor}(b), F_{c}(c) + \Phi_a(b))
\end{equation}
where $F_{c}: C \Arrow C'$ is a vector bundle morphism and $\Phi \in \Gamma(A^* \otimes B^* \otimes C')$.
\end{example}
\vspace{5pt}

A \textit{decomposition} for a double vector bundle $(D; A, B; M)$ is an isomorphism
$\sigma$ of double vector bundles from the trivial double vector bundle with core $C$ to $D$ covering
the identities on the side bundles $A, \, B$ and inducing the identity
on the core $C$. The space of decompositions for $D$ will be denoted
by $\dec{D}$. We recall now how this is an affine space over
$\Gamma(A^*\otimes B^*\otimes C)$. Given an element $\Phi \in
\Gamma(A^*\otimes B^*\otimes C)$, consider the double vector bundle morphism
\begin{equation}\label{I_nu}
\begin{array}{lccl}
I_{\Phi}:& A\oplus B \oplus C & \lrrow & A\oplus B \oplus C\\
        & (a,b,c) & \longmapsto & (a, b, c +\Phi_{a}(b))\\
\end{array}
\end{equation}
obtained from \eqref{trivial_morf} by taking $F_{\ver}, \, F_{\hor},
\, F_{c}$ to be the identity morphisms. For a decomposition $\sigma$,
\begin{equation}\label{affine_str}
\Phi \cdot \sigma := \sigma \circ I_{\Phi}
\end{equation}
defines the affine structure on $\dec{D}$. 

\begin{remark}\em
  The space of horizontal lifts is also affine over
  $\Gamma(A^*\otimes B^*\otimes C)$ (this follows directly from the
  definition of horizontal lifts). There is a natural one-to-one
  correspondence between decompositions and horizontal lifts for $D$
  \cite{GrMe10a}. Concretely, given a horizontal lift $h$, the
  decomposition $\sigma_h: A\oplus B \oplus C \Arrow D$ is given by
\begin{equation}\label{decomp}
  \sigma_h (a_m, b_m, c_m) = h(a)(b_m) \soma{B} (\leftexp{B}{0}_{\vphantom{1}_{b_m}} \soma{A} \overline{\,c_m}),
\end{equation}
where $m \in M$ and $a\in \Gamma(A)$ is any section with
$a(m)=a_m$. Conversely, given a decomposition $\sigma: A\oplus B\oplus
C \Arrow D$, the map $h_{\sigma}: \Gamma(A) \Arrow
\Gamma_{\ell}(B,D)$,
\begin{equation}\label{hor_decomp}
h_{\sigma}(a) (b_m) = \sigma(a(m), b_m, 0^C_m), \,\, m \in M,
\end{equation}
is a horizontal lift. The map $h \mapsto \sigma_h$ and its inverse
$\sigma \mapsto h_{\sigma}$ are affine.

\end{remark}

\begin{example}\label{tang_double} For a vector bundle $B \Arrow M$, 
$$
\begin{CD}
TB @>>> B\\
@VVV   @VVV\\
TM @>>> M
\end{CD}
$$
is a double vector bundle with core bundle $B \Arrow M$. The core section corresponding
to $b \in \Gamma(B)$ is the vertical lift $b^{\uparrow}: B \Arrow
TB$. One has that
$$
b^{\uparrow}(\ell_{\psi}) = \langle \psi, b \rangle \circ q_B \,\,\,\,\text{and} \,\,\,\,\, b^{\uparrow}(f \circ q_B) = 0,
$$
where $\ell_{\psi}, f\circ q_B \in C^{\infty}(B)$ are the linear
function and the pullback function corresponding to $\psi \in
\Gamma(B^*)$ and $f \in C^{\infty}(M)$, respectively. An element of
$\Gamma_{\ell}(B, TB)$ is called a \textit{linear vector field}. It is
well-known (see e.g.~\cite{Mackenzie05}) that a linear vector field $X:
B \Arrow TB$ covering $x: M \Arrow TM$ corresponds to a derivation
$L: \Gamma(B^*) \Arrow \Gamma(B^*)$ having $x$ as its symbol. The
precise correspondence is given by
$$
X(\ell_{\psi}) = \ell_{L(\psi)} \,\,\,\, \text{ and }  \,\,\, X(f \circ q_B)= \Lie_x(f) \circ q_B.
$$
Hence, the choice of a horizontal lift for $(TB; TM, B; M)$ is
equivalent to the choice of a connection on $B^*$. For convenience, we
shall prefer working with the dual connection on $B$ (see
\eqref{dual_conn}). In this case, one can identify $\dec{TB}$ with the
space of connections on $B$.
\end{example}

\begin{example} Let $A \Arrow M$ be a vector bundle and consider
  $TA \Arrow TM$ as the horizontal side bundle of the tangent double,
$$
\begin{CD}
TA @>>> TM\\
@VVV   @VVV\\
A @>>> M.
\end{CD}
$$
For any $a \in \Gamma(A)$, $Ta: TM \Arrow TA$ is a linear section
covering $a$ itself.  Yet, the map $a \mapsto Ta$ splits
\eqref{fat_seq} only at the level of sections, as it fails to be
$C^{\infty}(M)$-linear. The choice of a connection $\nabla$ on $A$
restores the $C^{\infty}(M)$-linearity and induces a horizontal lift
by
\begin{equation}
h(a)(x) = Ta(x) \soma{TM} (T0(x) - \overline{\nabla_x a}), \,\, x \in TM, \, a \in \Gamma(A).
\end{equation}
The associated decomposition $\sigma_h \in \dec{TA}$ coincides with
the one induced by $\nabla$ as in Example \ref{tang_double}.
\end{example}


\subsection{Dualization of double vector bundles.}
Given a double vector bundle $(D; A, B; M)$ with core $C$, its
\textit{horizontal dual} is the double vector bundle
\begin{equation}\label{hor_dual}
\begin{CD}
D_B^*@> {p_B} >> B\\
@Vp^{\hor}_{\vphantom{C}_{C^*}}VV     @VVq_BV\\
C^* @>>q_{C^*}>  M,
\end{CD}
\end{equation}
where $p_B: D_B^* \Arrow B$ is the dual of $q_B^D: D\Arrow B$ and, for $\xi \in (p_B)^{-1}(b_m)$,
\begin{equation}\label{C_proj}
  \langle p^{\hor}_{\vphantom{C}_{C^*}}(\xi), c_m \rangle = \langle \xi, \leftexp{B}{0}_{b_m} \soma{A} \overline{c_m} \rangle.
\end{equation}
The core bundle of $D_B^*$ is $A^*\Arrow M$. Similarly, the
\textit{vertical dual} is the double vector bundle
\begin{equation}\label{ver_dual}
\begin{CD}
D_A^*@> p^{\ver}_{\vphantom{C}_{C^*}} >> C^*\\
@Vp_{A}VV     @VVq_{C^*}V\\
A @>>q_{A}>  M
\end{CD}
\end{equation}
with core $B^*\Arrow M$.

In the following, we are mostly interested in the horizontal dual. For
$\psi \in \Gamma(A^*)$, the corresponding core section $\widehat{\psi}
\in \Gamma_c(B, D_B^*)$ is just $(q_A^D)^*\psi$. In particular,
\begin{equation}\label{pair1}
\langle \widehat{\psi}, \,\widehat{c} \rangle = 0 
\end{equation}
for $c \in \Gamma(C)$ and
\begin{equation}\label{pair2}
 \langle \widehat{\psi},\, h(a) \rangle = \langle \psi, a \rangle \circ q_B 
\end{equation}
for $a \in \Gamma(A)$ and  any horizontal lift $h: A \Arrow \widehat{A}$. 

Given a decomposition $\sigma: A\oplus B\oplus C \Arrow D$, the
inverse of its dual over $B$, $(\sigma_B^*)^{-1}: B \oplus C^*\oplus
A^* \Arrow D_B^* $, is a decomposition for $D_B^*$.

\begin{example}
  Let $B \Arrow M$ be a vector bundle and consider its tangent double
  $(TB; TM, B; M)$. The projection of the cotangent bundle $T^*B$ to
  $B^*$ is given, for $\xi \in T_{b_m}^*B$, by
$$
\langle p^{\hor}_{B^*}(\xi), c_m \rangle = \left\langle \xi,
  \left.\frac{d}{dt}\right|_{t=0}(b_m +t c_m) \right\rangle, \,\,
\text{ for } c_m \in B_m, \, m \in M.
$$
Given a decomposition $\sigma: TM\oplus B\oplus B \Arrow TB$, let
$\nabla$ be the corresponding connection on $B$. The inverse of the
dual of $\sigma$ over $B$ induces a horizontal lift $h: \Gamma(B^*)
\Arrow \Gamma_{\ell}(B, T^*B)$ given by
$$
h(\psi)(b_m) = (\sigma_B^*)^{-1}(b_m, \psi(m), 0_m^{TM}) =
d\ell_{\psi}(b_m) - \langle \nabla_{Tq_B(\cdot)} \,\psi, b_m
\rangle\in T_{b_m}^*B,
$$
where $\ell_{\psi} \in C^{\infty}(B)$ is the linear function
corresponding to $\psi \in \Gamma(B^*)$.
\end{example}



\section{$\mathcal{VB}$-algebroids and morphisms.}

Gracia-Saz and Mehta show in \cite{GrMe10a} how representations up to homotopy of a
Lie algebroid on a 2-term graded vector bundle encode the Lie
algebroid structures of \emph{$\VB$-algebroids}. These are double vector bundles with some
additional Lie algebroid structure that is compatible with the double
vector bundle
structure. In this section, we recall this correspondence and show how
it can be extended to morphisms. We also
check in Appendix \ref{dual} that it behaves well under dualization.

\subsection{$\VB$-algebroids.}

We begin with the definition of $\VB$-algebroids. We follow
\cite{GrMe10a} in our treatment of the subject.

\begin{definition}\label{VB_def}
  Let $(D; A, B; M)$ be a double vector bundle. We say that $(D \Arrow B; A \Arrow M)$
  is a $\VB$-algebroid if $D \Arrow B$ is a Lie algebroid, the anchor
  $\rho_D: D \Arrow TB$ is a bundle morphism over $\rho_A: A \Arrow
  TM$ and the three Lie bracket conditions below are satisfied:
\begin{itemize}
\item[(i)] $[\Gamma_{\ell}(B, D), \Gamma_{\ell}(B,D)]_D \subset \Gamma_{\ell}(B,D)$;
\item[(ii)] $[\Gamma_{\ell}(B,D), \Gamma_c(B,D)]_D \subset \Gamma_c(B,D)$;
\item[(iii)] $[\Gamma_c(B,D), \Gamma_c(B,D)]_D = 0$.
\end{itemize}
\end{definition}

A $\VB$-algebroid structure on $(D; A, B; M)$ naturally induces a Lie
algebroid structure on $A$ by taking the anchor to be $\rho_A$ and the
Lie bracket $[\cdot, \cdot]_A$ defined as follows: if $\X, \Y \in
\Gamma_{\ell}(B,D)$ cover $\X_0, \Y_0 \in \Gamma(A)$ respectively,
then $[\X, \Y]_D \in \Gamma_{\ell}(B,D)$ covers $[\X_0, \Y_0]_A \in
\Gamma(A)$. We call $A$ \textit{the base Lie algebroid} of $D$.

The next result from \cite{GrMe10a} relates VB-algebroid structures on
trivial double vector bundles and representations up to homotopy. 
Note that Arias Abad and Crainic show 
a related result on the relationship between representations up to homotopy and Lie algebroid
extensions \cite[Proposition 3.9.]{ArCr12}.

\begin{proposition}\label{rep_vb}
  Let $(A, \rho_A, [\cdot,\cdot]_A)$ be a Lie algebroid over $M$. Let
  $B \Arrow M$ and $C \Arrow M$ be vector bundles. There is a
  one-to-one correspondence between $\VB$-algebroid structures on the
  trivial double vector bundle $A\oplus B\oplus C$ with core $C$ and $A$ as side Lie algebroid, and
  2-term representations up to homotopy of $A$ on $V=C_{[0]}\oplus
  B_{[1]}$.
\end{proposition}

Let us give an explicit description of the $\VB$-algebroid structure
on $D=A\oplus B\oplus C$ corresponding to a 2-term representation
$(\partial, \nabla, K)$ of $A$ on $C_{[0]}\oplus B_{[1]}$. For $a \in
\Gamma(A)$, let $h: \Gamma(A) \hookrightarrow \Gamma_{\ell}(B, D)$ be
the canonical inclusion of Example \ref{trivial_dvb}.  Define as
follows the
anchor of $D$, $\rho_D: D\Arrow B$, on linear and core sections:
\begin{equation}\label{D_anchor}
\rho_D(h(a)) = X_{\nabla_a^1}, \,\, \rho_D(\widehat{c}) = \partial(c)^{\uparrow},
\end{equation}
where $X_{\nabla_a^1}, \, \partial(c)^{\uparrow} \in \mathfrak{X}(B)$
are, respectively, the linear vector field corresponding to the
derivation ${\nabla_a^1}^*: \Gamma(B^*) \Arrow \Gamma(B^*)$ and the vertical
vector field corresponding to $\partial(c) \in \Gamma(B)$ (see Example
\ref{tang_double}). The Lie bracket $[\cdot, \cdot]_D$ on $\Gamma(D)$
is given by the formulas below:
$$
[\widehat{c}_1, \widehat{c}_2]_D = 0
$$
\begin{equation}\label{c}
[h(a), \widehat{c}\,]_D = \widehat{\nabla^0_a \,c}, 
\end{equation}
and
\begin{equation}\label{curvature}
[h(a_1), h(a_2)]_D =  h([a_1, a_2]_A)  + \widehat{K}(a_1, a_2)
\end{equation}
where $a, a_1, a_2 \in \Gamma(A)$ and $c, c_1, c_2 \in \Gamma(C)$ and
$\widehat{K}(a_1, a_2) \in \Gamma_{\ell}(B, D)$ is the linear section
given by \eqref{core_morf}.

\begin{remark}\label{gen_remark}\em
  A $\VB$-algebroid structure on a general double vector bundle $(D;
  A, B; M)$ induces a representation up to homotopy of the base Lie
  algebroid $A$ on $C_{[0]}\oplus B_{[1]}$ once a decomposition
  $\sigma: A\oplus B\oplus C \Arrow D$ is chosen. The structure
  operators $(\partial, \nabla, K)$ are obtained from exactly the same
  formulas \eqref{D_anchor}, \eqref{c} and \eqref{curvature} by taking
  $h: \Gamma(A) \Arrow \Gamma_{\ell}(B, D)$ as the horizontal lift
  corresponding to $\sigma$. The isomorphism class of this
  representation does not depend on the choice of the
  decomposition. More precisely, if $\widetilde{\sigma}$ is another
  decomposition, then $\widetilde{\sigma}= \Phi \cdot \sigma$, for
  some $\Phi \in \Gamma(A^*\otimes B^*\otimes C)$ and the structure
  operators $(\widetilde{\partial}, \widetilde{\nabla},
  \widetilde{K})$ corresponding to $\widetilde{\sigma}$ are given by
\begin{equation}\label{new1}
\widetilde{\partial}=\partial;
\end{equation}
\begin{equation}\label{new2}
  \widetilde{\nabla}^0_a = \nabla^0_a - \Phi_a \circ \partial \text{ and }  \widetilde{\nabla}^1_a = \nabla^1_a - \partial \circ \Phi_a;
\end{equation}
\begin{equation}\label{new3}
\widetilde{K}(a,b) = K(a,b) + d_{\nabla^{\Hom}}\Phi(a,b) + \Phi_b\circ \partial \circ \Phi_a - \Phi_a\circ \partial \circ \Phi_b,
\end{equation}
for $a, b \in \Gamma(A)$. Moreover,
$(\mathrm{id}_C, \mathrm{id}_B, \Phi)$ are the components of an
$\Omega(A)$-linear isomorphism $\Omega(A, C_{[0]}\oplus B_{[1]})
\Arrow \Omega(A, C_{[0]}\oplus B_{[1]})$ which intertwines
$\mathcal{D}$ and $\widetilde{\mathcal{D}}$ (see \cite{GrMe10a} for more
details).
\end{remark}

The next two Examples recall how the double and the adjoint
representation arise in this way from $\VB$-algebroids.

\begin{example}\label{double_vb}
  The tangent double $(TB; TM; B; M)$ of a vector bundle $B \Arrow M$
  is canonically endowed with a $\VB$-algebroid structure $(TB \Arrow
  B; TM \Arrow M)$ with $\rho_{TB}= \mathrm{id}_B$,
  $\rho_{TM}=\mathrm{id}_{TM}$ and $[\cdot, \cdot]_{TB}$ given by the
  Lie bracket of vector fields. A horizontal lift $h: TM \Arrow
  \widehat{TM}$ is equivalent to a connection $\nabla:
  \Gamma(TM)\times \Gamma(B) \Arrow \Gamma(B)$ (see Example
  \ref{tang_double}). Equations \eqref{D_anchor} and \eqref{c} imply
  the equality $\partial= \mathrm{id}_B$ and show that the connection
  on $B_{[0]}\oplus B_{[1]}$ is given by $\nabla$ in degree $0$ and
  $1$. The equality $K = - R_{\nabla}$, with $R_\nabla$ the curvature
  of $\nabla$, follows from \eqref{curvature}. Hence, the element on
  $\mathrm{Rep}^{2}(TM)$ associated to $(TB\Arrow B; TM \Arrow M)$ is
  the isomorphism class of the double representation of $TM$ on
  $B\oplus B$ (see Example \ref{double}).
\end{example}

\begin{example}\label{adjoint_vb}
  Let $(A, \rho_A, [\cdot, \cdot]_A)$ be a Lie algebroid over $M$. The
  \emph{tangent prolongation} $(TA; A, TM; M)$ of $A$ has a $\VB$-algebroid
  structure $(TA \Arrow TM; A \Arrow M)$. We refer to \cite{Mackenzie05}
  for more details about this. Gracia-Saz and Mehta show that the element on
  $\mathrm{Rep}^{2}(A)$ associated to such a $\VB$-algebroid structure
  is exactly the adjoint representation of $A$ \cite{GrMe10a}.
\end{example}

Given a $\VB$-algebroid $(D \Arrow B; A \Arrow M)$, one can prove (see
\cite{Mackenzie11}) that the vertical dual $(D_A^* \Arrow C^*; A \Arrow
M)$ is a $\VB$-algebroid. By choosing a decomposition $\sigma \in
\dec{D}$, the inverse of its dual over $A$, $(\sigma_A^*)^{-1}$, is a
decomposition for $D_A^*$. In Appendix \ref{dual}, we prove that the
representations up to homotopy associated to $\sigma$ and
$(\sigma_A^*)^{-1}$ are dual to each other.

\begin{example}
  Let $(A, [\cdot, \cdot]_A, \rho_A)$ be a Lie algebroid over $M$. By
  Proposition \ref{dual_prop}, the $\VB$-algebroid structure of $(T^*A
  \Arrow A^*; A \Arrow M)$ obtained from taking the vertical dual of
  the tangent prolongation $(TA \Arrow TM; A \Arrow M)$ gives rise to
  the coadjoint representation $\mathrm{ad}^{\top} \in \Rep^2(A)$, the
  isomorphism class of the representation up to homotopy dual to the
  adjoint representation of $A$. We refer to \cite{Mackenzie05} for
  more details concerning the \textit{cotangent Lie algebroid} $T^*A
  \Arrow A^*$.
\end{example}


\subsection{Lie algebroid differential.}
Let $(A, \rho_A, [\cdot, \cdot]_A)$ be a Lie algebroid over $M$. Given
a 2-term representation up to homotopy of $A$ on $V= C_{[0]}\oplus
B_{[1]}$, we investigate how the Lie algebroid differential $d_D$ of
$D \Arrow B$, where $D=A\oplus B\oplus C$, is related to the structure
operators $(\partial, \nabla, K)$.
 Let $h: \Gamma(A) \hookrightarrow
\Gamma_{\ell}(B, D)$ be the natural inclusion considered in Example
\ref{trivial_dvb}.

First of all, it is straightforward to check that, for $f \in C^{\infty}(M)$, 
\begin{equation}\label{d_pull}
d_D(f\circ q_B)= (\rho_A\circ q_A^D)^*(df)= q_A^{D\, *}(d_A f),
\end{equation}
where $d_A$ is the Lie algebroid differential of $A$.

In the following, recall the identification $\Gamma_{\ell}(B, D^*_B) =
\Gamma(A^* \otimes B^*)\oplus \Gamma(C^*)$ (see Example
\ref{trivial_dvb}).

\begin{lemma}\label{d0_linear}
  Let $\ell_{\psi} \in C^{\infty}(B)$ be the linear function
  associated to $\psi \in \Gamma(B^*)$. The map $d_D(\ell_{\psi}): B
  \rmap D_B^*$ is a linear section given by
\begin{equation}\label{linear_zero}
d_D(\ell_{\psi}) = (d_{\nabla^*} \psi, \partial^* \psi)
\end{equation}
where $\nabla^*$ is the $A$-connection on $V^*=B^*_{[0]}\oplus
C^*_{[1]}$ dual to $\nabla$ and $d_{\nabla^*}: \Omega(A; V^*) \Arrow
\Omega(A;V^*)$ is the Koszul differential.
\end{lemma}

\begin{proof}
  The result that $d_D(\ell_{\psi})$ is a linear section follows from
  the fact that $d\ell_{\psi}: B \rmap T^*B$ is a linear section of
  the cotangent bundle (covering $\psi$ itself) and $\rho_D: D \Arrow
  TB$ is a double vector bundle morphism. Also, for $b_m \in B$, one has
\begin{align*}
  \langle p^{\hor}_{C^*}(d_D\ell_{\psi}), c_m \rangle =
  d\ell_{\psi}(\rho_D(\tilde{0}_{b_m} \soma{A} \overline{c_m} )) =
  \left.\frac{d}{dt}\right|_{t=0} \langle \psi(m), b_m +
  t \partial(c_m) \rangle = \langle \partial^*\psi(m), c_m \rangle.
\end{align*}
Finally, since $\rho_D(h(a))$ is the linear vector field on $B$ which
corresponds to the derivation $\nabla_a$ on $\Gamma(B)$, it follows
that
$$
\langle d_D(\ell_{\psi}), h(a) \rangle =
\Lie_{\rho_D(h(a))}\ell_{\psi} = \ell_{\nabla^{*}_a \, \psi}.
$$
Formula \eqref{linear_zero} follows immediately.
\end{proof}

Now let us consider the degree one part of $d_D$ (i.e.~$d_D:
\Gamma(D_B^*)\Arrow \Gamma(\wedge^2 D_B^*)$). As $D=q_B^!(A\oplus C)$
as a vector bundle over $B$, one has
\begin{equation}\label{wedge_decomp}
\wedge^2 D^*_B = q^{!}_B(\wedge^2 A^* \oplus (A^*\otimes C^*) \oplus \wedge^2 C^*).
\end{equation}

\begin{lemma}
  Choose $\psi\in\Gamma(A^*)$ and consider the corresponding core section
  $\widehat{\psi} \in \Gamma_c(B,D)$. With respect to the
  decomposition \eqref{wedge_decomp}, one has
\begin{equation}\label{psi_eq}
  d_D\widehat{\psi}: b_m \longmapsto (\,d_A \psi(m), \,0^{\scriptscriptstyle A^*\otimes C^*}_m , \, 0^{\scriptscriptstyle \wedge^2 C^*}_m).
\end{equation}
\end{lemma}

\begin{proof}
  Choose $a_1, \,a_2 \in \Gamma(A)$. As $\rho_D:D \Arrow TB$ is a
  vector bundle morphism over $\rho_A: A \Arrow TM$, one has that
  $Tq_B\circ \rho_D(h(a_i)) = \rho_A \circ q^D_A(h(a_i)) =
  \rho_A(a_i)$, for $i=1, 2$. Also, it follows from \eqref{pair2} that
  $\langle \widehat{\psi}, h(a_i) \rangle = \langle \psi, a_i \rangle
  \circ q_B$, for $i=1,2$. It is now straightforward to check that
$$
\begin{array}{rl}
  d_D \widehat{\psi}(h(a_1), h(a_2)) = & (d_A\psi(a_1,a_2))\circ q_B - \langle \widehat{\psi}, \,\widehat{K}(a_1,a_2) \rangle = (d_A\psi(a_1,a_2))\circ q_B.
\end{array}
$$
As for the component on $A^*\otimes C^*$, we have
$$
\begin{array}{rl}
  d_D\widehat{\psi}(h(a_1), \widehat{c}) 
  = & \Lie_{\rho_D(h(a_1))} \langle \widehat{\psi}, \,\widehat{c}
  \rangle -  \Lie_{\rho_D(\widehat{c})} \langle \widehat{\psi}, \,h(a_1)
  \rangle 
  - \langle \widehat{\psi}, \, [h(a_1), \widehat{c}\,]_D \rangle = 0.
\end{array}
$$
The first and the last term on the right hand side  vanish because of
\eqref{pair1}. Also, since $\rho_D(\widehat{c})$ is a vertical vector
field, it follows with \eqref{pair2} that the second term on the right
hand side
vanishes. One can prove in a similar manner that the component on
$\wedge^2 C^*$ is also zero.
\end{proof}

\begin{lemma}
  Let $Q\in \Gamma(A^* \otimes B^*)$ and $\gamma \in \Gamma(C^*)$.With
  respect to the decomposition \eqref{wedge_decomp}, we have
\begin{equation}\label{Q_diff}
  d_D(Q, 0): b_m \longmapsto (\langle d_{\nabla^*} Q, b_m \rangle,  - (\id_{A^*}\otimes \partial^*)(Q), \, 0^{\scriptscriptstyle\wedge^2 C^*}_m ),
\end{equation}
and
\begin{equation}\label{d_gamma}
  d_D(0, \gamma): b_m \longmapsto ( -\langle \, K^*\gamma, b_m \,\rangle, \, d_{\nabla^{*}} \gamma, \, 0^{\scriptscriptstyle\wedge^2 C^*}_m),
\end{equation}
\end{lemma}

\begin{proof}
Fix $a_1,a_2\in\Gamma(A)$ and $c\in\Gamma(C)$. We have
$$
d_D(Q,0)(h(a_1), h(a_2)) = \Lie_{\rho_D(h(a_1))} \,\ell_{Q(a_2)} -
\Lie_{\rho_D(h(a_2))} \,\ell_{Q(a_1)} -\left\langle (Q,0), \,[h(a_1),
  h(a_2)]_D \right\rangle.
$$
As $\rho_D(h(a_i)) \in \mathfrak{X}(B)$ is the linear vector field
corresponding to $\nabla_{a_i}$, we have
$$
 \Lie_{\rho_D(h(a_i))}\, \ell_{Q(a_j)} = \ell_{\nabla^*_{a_i} Q(a_j)}, \,\, \text{for }  1 \leq i\neq j \leq 2. 
$$
Also, by \eqref{pair1} and \eqref{curvature}, we get
$$
\left\langle (Q,0),\, [h(a_1), h(a_2)]_D^{\vphantom{A}}\right\rangle =
\left\langle (Q,0),\, h([a_1, a_2]_A)^{\vphantom{A}} \right\rangle =
\ell_{Q([a_1, a_2]_A)}.
$$
The formula for the component on $\wedge^2 A^*$ now follows by
assembling the terms. Similarly, using that $\langle (Q,0),
\widehat{c} \rangle = \langle (Q,0),\, [h(a_1), \widehat{c}\,]_D
\rangle = 0$, we get
$$
d_D(Q,0)(h(a_1), \widehat{c}) = -
\Lie_{\rho_D(\widehat{c})}\ell_{Q(a_1)} = \langle Q(a_1),
- \partial(c) \rangle \circ q_B.
$$
It is  straightforward to check now that the component in
$\wedge^{2}C^*$ is zero.  The proof of \eqref{d_gamma} is a similar computation
that we leave to the reader.
\end{proof}


\subsection{Morphisms.}

\begin{definition}
  Let $(D \Arrow B; A \Arrow M)$ and $(D' \Arrow B'; A' \Arrow M)$ be
  $\VB$-algebroids. A $\VB$-algebroid morphism from $D$ to $D'$ is a
  double vector bundle morphism $(F; F_{\ver}; F_{\hor})$ from $D$ to $D'$ such that
  $F$ is a Lie algebroid morphism.
\end{definition}

Our aim is to relate $\VB$-algebroid morphisms with morphisms of
representations up to homotopy. Using decompositions, it suffices
to consider morphisms $F$ between trivial double vector bundles $D=A\oplus B\oplus C$
and $D'=A'\oplus B'\oplus C'$. From Example \ref{trivial_dvb}, we know
that a double vector bundle morphism $F: D \Arrow D'$ is determined by vector bundle
morphisms $F_{\ver}: A \Arrow A'$, $F_{\hor}: B \Arrow B'$, $F_{c}: C
\Arrow C'$ and $\Phi \in \Omega^1(A, \Hom(B, C')).$

\begin{theorem}\label{main}
  $F: D \Arrow D'$ is a $\VB$-morphism if and only if $F_{\ver}: A
  \Arrow A'$ is a Lie algebroid morphism and $(F_{c}, F_{\hor}, \Phi)$
  are the components of a morphism $(A, V) \Rightarrow (A', W)$ over
  $F_{\ver}$ between the associated representations up to homotopy
  $V=C_{[0]}\oplus B_{[1]} \in \R \mathrm{ep}^2(A)$ and
  $W=C'_{[0]}\oplus B'_{[1]}\in \R \mathrm{ep}^2(A')$.
\end{theorem}

\begin{remark}
Combining the results in \cite{GrMe10a} with Theorem \ref{main} we conclude that the category of $2$-term representations up to homotopy of a Lie algebroid $A$ is equivalent to the category
of $\VB$-algebroids with side algebroid $A$.
\end{remark}

In the following, let $d_D$ and $d_{D'}$ be the Lie algebroid
differentials of $D \Arrow B$ and $D'\Arrow B'$, respectively. Recall
that $F: D \Arrow D'$ is a Lie algebroid morphism if and only if the
associated exterior algebra morphism, $F^*: \Gamma(\wedge^{\bullet}
{D'}_{B'}^*) \Arrow \Gamma(\wedge^{\bullet} D_B^*)$, intertwines $d_D$
and $d_{D'}$\footnote{Recall that $F^*$ is defined by
$(F^*\omega)(b)(d_1(b),\ldots, d_n(b))=\omega_{F_{\hor}(b)}(F(d_1(b)),
\ldots, F(d_n(b)))$ for $\omega\in \Gamma(\wedge^{n}
{D'}_{B'}^*)$, $b\in B$ and $d_1,\ldots,d_n\in\Gamma(B,D)$.
In particular, we have $F^*(g)= g \circ F_{\hor}$,
for $g\in C^\infty(B)=\Gamma(\wedge^{0}
{D'}_{B'}^*)$.}. Theorem
  \ref{main} will follow from the thorough study of the relation
  $F^*\circ d_{D'} = d_D \circ F^*$, which we carry on in Lemmas
  \ref{lem2} and \ref{lem3} below. First, we shall need a Lemma which
  gives useful formulas for $F^*$ in degree 1, $ F^*: \Gamma({D'}^*_{B'})
  \Arrow \Gamma(D_B^*). $ 

\begin{lemma}\label{dual_morf}
Choose $Q\in \Gamma(B'^*\otimes A'^*)$, $\gamma\in\Gamma(C'^*)$ and
$\psi\in\Gamma(A'^*)$. Then $F^* \widehat{\psi}=
\widehat{F_{\ver}^*\psi}$ and $F^*(Q, \gamma) = (F_{\ver}^*\otimes
F_{\hor}^*(Q) + \langle \Phi, \gamma \rangle, \, F_{c}^*\gamma)$.
\end{lemma}

\begin{proof}
The result follows directly from Example \ref{trivial_dvb}. We leave the
details to the reader.
\end{proof}

Let now $(\partial_W, \nabla^W, K_W)$ and $(\partial_V, \nabla^V,
K_V)$ be the structure operators of the representations up to homotopy
of $A$ on $V=
C_{[0]}\oplus B_{[1]}$ and of $A'$ on $W= C'_{[0]} \oplus B'_{[1]}$,
respectively, and let $\nabla^{\Hom}$ be the $A$-connection on
$\uHom(V, W)$ obtained from $\nabla^V$ and
$(\nabla^W)^{\scriptscriptstyle F_{\ver}}$ (see \eqref{pull_con}).

\begin{lemma}\label{lem2}
$F^* \circ d_{D'} = d_{D} \circ F^*$ 
holds on $\Gamma(\wedge^0 {D'}^*_{B'})=C^\infty(B')$ if and only if
\begin{equation}\label{eq1_lem2}
\rho_{A'} \circ F_{\ver} = \rho_A,
\end{equation}
\begin{equation}\label{eq2_lem2}
F_{\hor} \circ \partial_V = \partial_W \circ F_{c}
\end{equation}
and
\begin{equation}\label{eq3_lem2}
\nabla^{\Hom}_a F_{\hor}  = \partial_W \circ \Phi_a, \, \forall\, a \in \Gamma(A).
\end{equation}
\end{lemma}

\begin{proof}
  It suffices to consider $f \circ q_{B'}$, for $f
  \in C^{\infty}(M)$ and linear functions $\ell_{\beta}$, for $\beta
  \in \Gamma({B'}^*)$. Now, \eqref{eq1_lem2} follows directly from
  \eqref{d_pull}. For the other two equations, first observe that
  $F^*(\ell_{\beta}) = \ell_{F_{\hor}^*\beta}$. 
The identity $$
d_{D}(\ell_{F_{\hor}^*\beta})= (d_{{\nabla^V}^*} F_{\hor}^*\,\beta , (F_{\hor}
\circ \partial_V)^*\beta) \in \Gamma(A^*\otimes B^*)\oplus \Gamma(C^*)
$$
follows from
  \eqref{linear_zero}.
On the other hand, due to \eqref{linear_zero} and Lemma
\ref{dual_morf}, we have $F^*(d_{D}(\ell_{\beta})) = (Q, \gamma)$,
where
$$
Q = F^*_{\ver} \otimes F^*_{\hor} \,(d_{{\nabla^W}^*}\beta) + \langle \Phi, \partial_W^* \beta \rangle
$$
$$
\gamma= (\partial_W\circ F_{c})^*\beta.
$$
By comparing the components in $\Gamma(C^*)$ and $\Gamma(A^*\otimes
B^*)$, one finds equations which are dual to
\eqref{eq2_lem2} and \eqref{eq3_lem2}, respectively. This proves the
lemma.
\end{proof}

\begin{lemma}\label{lem3}
$F^* \circ d_{D'} = d_{D} \circ F^*$ holds on $\Gamma(\wedge^1 D'^*_{B'})$ if and only if
\begin{equation}\label{eq1_lem3}
d_A \circ F^*_{\ver} = F^*_{\ver} \circ d_{A'} \text{ in } \Gamma(A'^*),
\end{equation}
\begin{equation}\label{eq2_lem3}
\nabla^{\Hom}_a F_{c} = \Phi_a\circ \partial_V, \, \forall \,a \in \Gamma(A)
\end{equation}
and
\begin{equation}\label{eq3_lem3}
d_{\nabla^{\Hom}}\Phi= F_{c} \circ K_V - (F_{\ver}^*K_W) \circ F_{\hor}.
\end{equation}
\end{lemma}

\begin{proof}
  It suffices to consider core sections $\widehat{\psi}$ and linear
  sections of the type $(0, \gamma)$, where $\gamma \in \Gamma(C'^*)$
  and $\psi \in \Gamma(A'^*)$. Equation \eqref{eq1_lem3} is equivalent
  to $F^* \circ d_{D'}(\widehat{\psi}) = d_D \circ
  F^*(\widehat{\psi})$. Now, according to the decomposition
  \eqref{wedge_decomp}, we find $F^* \circ d_{D'}(0, \gamma) =
  (\Lambda_1, \Lambda_2, \Lambda_3)$, where $\Lambda_3$ is the zero
  section of $\wedge^2 q_{B}^!C^*$,
$$
\Lambda_1(b_m) = -\langle \,(F_{\ver}^*K_W)^*\gamma, F_{\hor}(b_m)
\,\rangle + (F_{\ver}^* \wedge \Phi^*(b_m))(d_{{\nabla^W}^*} \gamma(m)),
$$
and 
$$
\Lambda_2(b_m)= F_{\ver}^* \otimes F_{c}^*\,(d_{{\nabla^W}^*} \gamma(m)),
$$
where $m \in M$, $b_m \in B_m$ and $\Phi(b_m)$ is seen as a map from
$A$ to $C'$ with dual $\Phi^*(b_m): C'^* \Arrow A^*$.
Similarly, by Lemma \ref{dual_morf} and formulas \eqref{Q_diff},
\eqref{d_gamma}, it follows that $d_D(F^*(0, \gamma)) = d_D(\langle
\Phi, \gamma\rangle, \,F_{c}^*\gamma) = (\Theta_1, \Theta_2,
\Theta_3)$, where $\Theta_3$ is again the zero section of $\wedge^2
q_{B}^!C^*$,
$$
\Theta_1(b_m) = \langle d_{{\nabla^V}^*} \langle \Phi, \gamma \rangle, \, b_m
\rangle - \langle K_V^*(F_{c}^*\gamma), b_m \rangle.
$$
and
$$
\Theta_2(b_m) = - (\id_{A}^* \otimes \partial_V^*)\langle \Phi,
\gamma(m) \rangle + d_{{\nabla^V}^*}  F_{c}^*\gamma(m).
$$
The equalities $\Lambda_2 = \Theta_2$ and $\Lambda_1= \Theta_1$ are
equivalent to the equations dual to \eqref{eq2_lem3} an \eqref{eq3_lem3},
respectively.
\end{proof}

\begin{proof}[Proof of Theorem \ref{main}]
  Equations \eqref{eq1_lem2} and \eqref{eq1_lem3} are equivalent to
  $F_{\ver}$ being a Lie algebroid morphism. Similarly, equations
  \eqref{eq2_lem2}, \eqref{eq3_lem2}, \eqref{eq2_lem3} and
  \eqref{eq3_lem3} are equivalent to $(F_{c}, F_{\hor}, \Phi)$ being
  the components of a morphism $(A, V) \Rightarrow (A', W)$. This
  proves the Theorem.
\end{proof}


\begin{example}
  Let $(A, [\cdot, \cdot]_A, \rho_A)$ be a Lie algebroid over $M$. An
  IM-2-form \cite{BuCrWeZh04} on $A$ is a pair $(\mu, \nu)$ where $\mu: A \Arrow T^*M$
  and $\nu: A \Arrow \wedge^2 T^*M$ such that
\begin{enumerate}
\item$\langle \mu(a), \rho_A(b) \rangle = - \langle \mu(b), \rho_A(a) \rangle$;
\item$\mu([a,b]) = \Lie_{\rho_A(a)} \mu(b) - i_{\rho_A(b)}(d\mu(a) + \nu(a))$;
\item$\nu([a,b]) = \Lie_{\rho_A(a)}\nu(b) - i_{\rho_A(b)}d\nu(a)$,
\end{enumerate}
for $a, b \in \Gamma(A)$. In \cite{BuCa12} (see also \cite{BuCaOr09} for the
case where $\nu=0$), it is shown that every IM-2-form is associated to
a 2-form $\Lambda \in \Omega^2(A)$ whose associated map
$\Lambda_{\sharp}: TA \Arrow T^*A$ is a $\VB$-algebroid morphism from
$(TA \Arrow TM; A \Arrow M)$ to $(T^*A \Arrow A^*; A \Arrow M)$
inducing $\mu: A \Arrow T^*M$ on the core bundles and $-\mu^*:TM
\Arrow A^*$ on the side bundles.
 Let $\sigma \in \dec{TA}$ and
$\sigma_A^*$ be its dual over the side bundle $A$. From \cite{BuCa12}
(see Lemma 3.6 there), it follows that $F= \sigma_A^* \circ \Lambda_\sharp
\circ \sigma: A\oplus TM \oplus A \Arrow A\oplus A^*\oplus T^*M$ has
components given by $F_{\ver} = \mathrm{id}_A$, $F_{\hor}=
-\mu^*$, $F_c = \mu$ and
$$
\Phi = \nu + d_{\nabla^*}\mu^* \in \Omega^1(A; \Hom(TM, T^*M)).
$$
where $\nabla$ is the connection on $A$ associated to $\sigma$ and
$d_{\nabla^*}: \Omega(TM; A^*) \Arrow \Omega(TM; A^*)$ the Koszul
differential associated to the dual connection. Note that we are
identifying $\Omega^2(TM, A^*)$ with $\Omega^1(A; \wedge^2 T^*M)$ and
seeing $\wedge^2 T^*M$ as a subset of $\Hom(TM, T^*M)$. So, as a result of
Theorem \ref{main}, one has that $(\mu, \nu)$ is an IM-2-form if and
only if $(\mu, -\mu^*, \nu + d_{\nabla^*}\mu^*)$ are the components of a
morphism from the adjoint representation $ad_{\nabla}(A)$ to the
coadjoint representation $ad^{\top}_{\nabla}(A)$.
\end{example}

\begin{example}
  Let $(A, [\cdot, \cdot], \rho_A)$ be a Lie algebroid such that its
  dual $A^*$ has also a Lie algebroid structure $(A^*, [\cdot,
  \cdot]_{A^*}, \rho_{A^*})$. It is shown in \cite{MaXu94} that $(A,
  A^*)$ is a Lie bialgebroid if and only if $\pi_A^{\sharp}: T^*A
  \Arrow TA$ is a $\VB$-algebroid morphism from the cotangent Lie
  algebroid $(T^*A \Arrow A^*, A \Arrow M)$ to the tangent
  prolongation $(TA \Arrow TM; A \Arrow M)$, where $\pi_A\in
  \Gamma(\wedge^2 TA)$ is the linear Poisson bivector corresponding to
  the Lie algebroid $A^*$. For any decomposition $\sigma \in
  \dec{TA}$, it follows from \cite{MaXu94} (see Corollary 6.5 there)
  that $ \sigma^{-1} \circ \pi_A^{\sharp} \circ (\sigma_A^*)^{-1}:
  A\oplus A^*\oplus T^*M \Arrow A\oplus TM \oplus A $ has components
  given by $F_{\ver} = \mathrm{id}_A$, $F_{\hor} = \rho_{A^*}$, $F_{c}
  = - \rho_{A^*}^*$ and $\Phi \in \Omega^1(A; \Hom(A^*, A))$ defined
  by
$$
\langle \Phi_{a}(\alpha), \beta \rangle = -d_{A^*} a (\alpha, \beta) +
\langle \beta, \nabla_{\rho_{A^*}(\alpha)} \,a \rangle - \langle
\alpha, \nabla_{\rho_{A^*}(\beta)} \, a \rangle , \,\, (\alpha, \beta)
\in A^*\times_M A^*,
$$
where $d_{A^*}: \Gamma(\wedge A) \Arrow \Gamma(\wedge A)$ is the Lie
algebroid differential of $A^*$, $\nabla: \Gamma(TM) \times\Gamma(A)
\Arrow \Gamma(A)$ is the connection corresponding to $\sigma$ and
$\sigma_A^*$ is the dual of $\sigma$ over $A$. Note that
$$
\langle \Phi_{a}(\alpha), \beta) \rangle = \langle a, T^{\rm bas} (\alpha, \beta) \rangle,
$$
where $T^{\rm bas}$ is the torsion of the basic connection 
$$
\begin{array}{ccl}
\Gamma(A^*)\times \Gamma(A^*) & \lrrow & \Gamma(A^*)\\
    (\alpha, \beta) &  \longmapsto & [\alpha, \beta]_{A^*} + \nabla^*_{\rho_{A^*}(\beta)} \,\alpha.
\end{array}
$$
As a result of Theorem \ref{main}, we have that $(A, A^*)$ is a Lie
bialgebroid if and only if $(-\rho^*_{A^*}, \rho_{A^*}, T^{\rm bas})$
are the components of a morphism from the coadjoint representation
$\mathrm{ad}_{\nabla}^{\top}$ to the adjoint representation
$\mathrm{ad}_{\nabla}$.

\end{example}


\section{Distributions and foliations.}


Let $q_B: B \Arrow M$ be a vector bundle. A \textit{linear
  distribution} on $B$ is a subbundle $\Delta \subset TB$ such that
\begin{equation}\label{dist}
\begin{CD}
\Delta @>>>   B\\
 @V Tq_B VV   @VVq_BV\\
\Delta_M @>>> M
\end{CD}
\end{equation}
is a double vector bundle. It is called a \textit{linear foliation} if
$\Delta$ is integrable (or, equivalently, $\Delta \Arrow B$ is a Lie
subalgebroid of $TB \Arrow B$). Linear distributions and foliations
are particular examples of double vector subbundles and
$\VB$-subalgebroids, respectively. In this section we develop the
general theory of these objects. Our goal is to identify
invariants of distributions and foliations on Lie algebroids.

\subsection{Double vector subbundles and adapted decompositions.}

\begin{definition}
  Let $(D'; A', B'; M)$ be a double vector bundle. We say that $(D; A, B; M)$ is a
  \textit{double vector subbundle} of $D'$ if
\begin{enumerate}
\item $(D; A; B; M)$ is a double vector bundle;
\item $D \subset D'$; $A \subset A'$ and $B \subset B'$ are subbundles;
\item the inclusion $i: D \hookrightarrow D'$ is a morphism of double vector bundles
  inducing the inclusions $i_A: A\hookrightarrow A'$ and $i_B:
  B\hookrightarrow B'$ on the side bundles.
\end{enumerate}
\end{definition}

Note that the core $C'$ of $D'$ is a subbundle of $C$ and the map $i:
D \hookrightarrow D'$ induces the inclusion $i_C: C' \hrrow C$ on the
core bundles.

\begin{example}\label{model_sub}
Let $D'=A'\oplus B'\oplus C'$ be the trivial double vector bundle with core $C'$. Given
vector subbundles $A\subset A'$, $B\subset B'$ and $C\subset C'$, the
trivial double vector bundle $D=A\oplus B\oplus C$
with core $C$ is canonically a double vector subbundle of $D'$.
\end{example}

The inclusion of trivial double vector bundles of Example \ref{model_sub} should be
seen as a model for general double vector subbundles. Let us be more
precise.

\begin{definition}
  Let $(D;A,B;M)$ be a double vector sub-bundle of $(D'; A', B';
  M)$. We say that a decomposition $\sigma': A'\oplus B'\oplus C'
  \Arrow D'$ is \textit{adapted to} $D$ if $\sigma'(A\oplus B\oplus
  C)= D$. In this case, the induced decomposition
  $\sigma:=\sigma'|_{A\oplus B\oplus C}$ of $D$ makes the diagram
  below commutative
\begin{equation}\label{adapted_diagram}
  \xy
  (0,0)*++^{A\oplus B \oplus C}="sub"; (0,10)*+^{A'\oplus B'\oplus C'}="can"; (30,10)*+^{D'}="D'"; (30,0)*++^{D}="D";
  {\ar@{^{(}->}"sub"; "can"}; 
  {\ar@{->}_{\hspace{15pt}\sigma}"sub"; "D"};
  {\ar@{->}^{\hspace{15pt}\sigma'}"can"; "D'"};
  {\ar@{^{(}->}_{\,\,i}"D"; "D'"};
\endxy
\end{equation}
where the left vertical arrow is the canonical inclusion of Example \ref{model_sub}
\end{definition}

A horizontal lift $h: A' \Arrow \widehat{A'}$ is adapted to $D$ if its
corresponding decomposition $\sigma_h$ \eqref{decomp} is adapted to
$D$. Equivalently, $h$ is adapted to $D$ if and only if, for $a \in
\Gamma(A)$, the linear section $h(a): B' \Arrow D'$ satisfies $h(a)(B)
\subset D$. In this case, $h|_{A}$ is a horizontal lift for $D$.

\begin{example}
  A connection $\nabla: \Gamma(TM) \times \Gamma(B) \Arrow \Gamma(B)$
  is adapted to a linear distribution $\Delta$ if for every $x \in
  \Gamma(TM)$ the linear vector field $X_{\nabla_x}: B \Arrow TB$
  corresponding to the derivation $\nabla_x: \Gamma(B) \Arrow
  \Gamma(B)$ is a section of the distribution $\Delta$.
\end{example}

Recall that $\dec{D'}$, the space of decompositions of $D'$, is affine
modelled over $\Gamma(A'^*\otimes B'^*\otimes C')$. Define
\begin{equation}\label{sub_affine}
  \Gamma_{A, B, C}=\{\Phi \in \Gamma(A'^*\otimes B'^*\otimes C')\,\, | \,\, \Phi_a(B) \subset C, \, \forall \, a \in A\}.
\end{equation}

\begin{proposition}\label{adapted_prop}
  Let $(D'; A', B'; M)$ be a double vector bundle and $A,B$ and $C$ vector subbundles 
of $A',B'$ and $C'$, respectively. There is a one-to-one correspondence
$$
\vspace{2pt}
\left\{
\begin{array}{c}
\text{Double vector subbundles} \, (D; A, B; M) \text{ of } D' \\
\text{having } \, $C$ \text{ as core bundle.} 
\end{array}
\right\} \,\stackrel{1-1}\longleftrightarrow \, \frac{\dec{D'}}{\Gamma_{A,B,C}}.
\vspace{2pt}
$$
More precisely, for a double vector subbundle $(D;A,B;M)$, the
set of decompositions adapted to $D$ is an orbit for $\Gamma_{A, B,
  C}$. Reciprocally, given a decomposition $\sigma' \in \dec{D'}$, the
double vector subbundle $D=\sigma'(A\oplus B\oplus C)$ only depends on
the $\Gamma_{A, B, C}$-orbit of $\sigma'$ and any decomposition in
this orbit is adapted to $D$.
\end{proposition}

\begin{proof}
  For a double vector subbundle $(D; A, B; M)$, we shall first prove
  that decompositions adapted to $D$ always exist and then that they
  form an orbit for $\Gamma_{A, B, C}$. Begin with two arbitrary
  decompositions $\sigma: A\oplus B\oplus C \Arrow D$ and $\sigma':
  A'\oplus B'\oplus C' \Arrow D'$ and consider $\sigma'^{-1} \circ
  i\circ \sigma: A \oplus B \oplus C \Arrow A'\oplus B'\oplus C'$. It
  is a morphism between trivial double vector bundles inducing the inclusions on $A$,
  $B$ and $C$. Hence, there exists $\Phi \in \Gamma(A^*\otimes
  B^*\otimes C')$ such that
$$
\sigma'^{-1}\circ i \circ \sigma (a, b, c) = (a, b, c + \Phi_a(b)).
$$
For any $\Phi' \in \Gamma(A'^*\otimes B'^*\otimes C')$ extending
$\Phi$ (i.e.~$\Phi'_a(b) = \Phi_a(b)$, $\forall\, a \in A, b \in B$),
the decomposition $\Phi'\cdot\sigma'$ is adapted to $D$. Now, if $\sigma_1,
\sigma_2$ are arbitrary decompositions, there exists an unique $\Phi
\in \Gamma(A'^*\otimes B'^* \otimes C')$ such that $\sigma_1 = \Phi
\cdot \sigma_2$. It is straightforward to check that they lie in the
same orbit if and only if $I_{\Phi}= \sigma_2^{-1}\circ \sigma_1$ (see
\eqref{I_nu}) preserves $A\oplus B\oplus C$. This implies that the
decompositions adapted to $D$ is an $\Gamma_{A, B, C}$-orbit and that
the map
$$
\frac{\dec{D'}}{\Gamma_{A,B,C}} \ni [\sigma '] \mapsto \sigma '(A\oplus B\oplus C)
$$
is well-defined.
\end{proof}

Now we will use Proposition \ref{adapted_prop} to classify linear
distributions. Let $(\Delta; \Delta_M, B;M)$ be a linear distribution
with core $C \subset B$ and consider the quotient map $\pi: B \Arrow
B/C$. In the following, we identify $\dec{TB}$ with the space of
connections $\nabla: \Gamma(B) \Arrow \Gamma(T^*M\otimes B)$.

\begin{lemma}\label{1_1:lemma}
The map
$$
\begin{array}{rcl}
\displaystyle \frac{\dec{TB}}{\Gamma_{\Delta_M, B, C}} & \lrrow & \left\{\bbd: \Gamma(B) \Arrow \Gamma(\Delta_M^*\otimes (B/C))\,\, 
\left|
\begin{array}{l}
 \bbd_x (fb) = f\bbd_x(b) + (\Lie_xf) \pi(b), \vspace{3pt}\\
 \forall \, f \in C^{\infty}(M), \, b \in \Gamma(B), \, x \in \Gamma(\Delta_M)
\end{array}
\right.
\right\}
\\
& &\\
 \left[\nabla\right] & \longmapsto & (r\otimes\pi)\circ \nabla
\end{array}
$$
is a bijection, where $r: T^*M \Arrow \Delta_M^*$ is the restriction map.
\end{lemma}

\begin{proof}
  The fact that the map is well-defined follows directly from the
  definition of the affine action \eqref{affine_str} on the space of
  connections. Let us now prove that given $\bbd: \Gamma(B) \Arrow
  \Gamma(\Delta_M^* \otimes (B/C))$ satisfying the Leibniz equation
\begin{equation}\label{Leibniz}
\bbd_x (fb) = f\bbd_x(b) + (\Lie_xf) \pi(b),
\end{equation}
there exists a connection $\nabla \in \dec{TB}$ such that
$\bbd=(r\otimes \pi) \circ \nabla$.  For this, let $s: B/C \Arrow B$
be a linear section for the quotient projection $\pi: B \Arrow B/C$
and identify $B$ with $(B/C)\oplus C$ using $s$. Also, choose a
connection $\widetilde{\nabla}: \Gamma(TM) \times \Gamma(B) \Arrow
\Gamma(B)$ which preserves both $B$ and $B/C$.

First, note that formula \eqref{Leibniz} implies that,
for $x \in \Gamma(\Delta_M)$, $\bbd_x: \Gamma(B) \Arrow \Gamma(B/C)$
is actually linear when restricted to $\Gamma(C)$. Second, note that \eqref{Leibniz} implies that the map 
$$
\begin{array}{rccl}
\bbd_x \circ s: & \Gamma(B/C) & \lrrow &  \Gamma(B/C)\\
  & \gamma & \longmapsto & \bbd_x(s(\gamma))\\
\end{array}
$$
is a derivation on $B/C$ having $x \in \Gamma(\Delta_M)$ as symbol. So, define $\widetilde{\Phi}
\in \Gamma(T^*M \otimes B^* \otimes B)$
$$
\widetilde{\Phi}_x(\gamma) = 
\left\{
\begin{array}{ll}
   \left\{
   \begin{array}{ll}
    s \circ \bbd_x(\gamma), & \text{if} \,\, \gamma \in C;\\
    (\bbd_x \circ s -\widetilde{\nabla}_x)(\gamma), & \text{if} \,\, \gamma \in B/C;
\end{array}
\right. & \text{ if } x \in \Gamma(\Delta_M); \vspace{5pt}\\
0,& \text{otherwise.}
\end{array}
\right.
$$
It is now straightforward to check that the connection $\nabla:
\Gamma(B) \Arrow  \Gamma(T^*M\otimes B)$ given by $\nabla =
\widetilde{\nabla} + \widetilde{\Phi}$ satisfies $\bbd= (r\otimes \pi)
\circ \nabla$.
\end{proof}

For a linear distribution $(\Delta; \Delta_M, B; M)$ with core $C$,
one can canonically construct a map $\bbd^{\Delta}: \Gamma(B) \Arrow
\Gamma(\Delta_M^* \otimes (B/C))$ satisfying \eqref{Leibniz} as
follows:
\begin{equation}\label{mu_delta}
\bbd^{\Delta}_x(b) = \pi\circ L_{X}(b), \,\, x \in \Gamma(\Delta_M), \, b \in \Gamma(B),
\end{equation}
where $X: B \Arrow \Delta$ is any linear section covering $x$,
$L_{X}: \Gamma(B) \Arrow \Gamma(B)$ is the derivation defined by
\begin{equation}\label{derivation}
L_{X}(b)^{\uparrow} = [X, b^{\uparrow}], \text{ for } b \in \Gamma(B).
\end{equation}
It is straightforward to check that $\pi \circ L_{X}$ depends only
on $x$ and not on the particular choice of $X: B \Arrow
\Delta$.

\begin{theorem}\label{spencer_thm}
  A connection $\nabla: \Gamma(B) \Arrow \Gamma(T^*M \otimes B)$ is
  adapted to $\Delta$ if and only if
\begin{equation}\label{adapted_cond}
(r\otimes \pi) \circ \nabla = \bbd^{\Delta}.
\end{equation}
In particular, the map $\Delta \mapsto \bbd^{\Delta}$ establishes a one-to-one correspondence
between linear distributions $(\Delta; \Delta_M, B; M)$ with core $C$
and $\R$-linear maps $\bbd: \Gamma(B) \Arrow \Gamma(\Delta_M^* \otimes
(B/C))$ satisfying the Leibniz equation \eqref{Leibniz}.
\end{theorem}

\begin{proof}
If   $\nabla$ is adapted to $\Delta$, then the linear vector field
$X_{\nabla_x}: B \Arrow TB$ corresponding to the derivation $\nabla_x: \Gamma(B)
 \Arrow \Gamma(B)$ is a linear section of $\Delta$. Hence, it follows 
from the definition \eqref{mu_delta} that
$$
\bbd^{\Delta}_x = \pi\circ L_{X_{\nabla_x}} = \pi \circ \nabla_x,
$$
for $x \in \Gamma(\Delta_M)$. On the other hand, if
\eqref{adapted_cond} holds, then, for every $x \in \Gamma(\Delta_M)$,
there exists a linear section $X: B \Arrow \Delta$ covering $x$
such that $\delta:=\nabla_x - L_{X} \in \Hom(B, C)$, where
$L_{X}$ is the derivation defined by \eqref{derivation}. In terms
of sections,
$$
X_{\nabla_x} = X + \delta^{\uparrow}.
$$
As $C$ is the core bundle of $\Delta$, one gets that $X_{\nabla_x}$
is a section of $\Delta$ and, therefore, $\nabla$ is adapted to
$\Delta$. The last statement follows from Lemma \eqref{1_1:lemma}.
\end{proof}

\begin{remark}\label{lin_mult}\em
  If one considers the Lie groupoid $B \rightrightarrows M$, where the
  source and the target are the projection $q_B:B \Arrow M$ and the
  multiplication is addition on the fibers, then a linear distribution
  is just a multiplicative distribution in the sense of
  \cite{CrSaSt12, JoOr13}. In this situation, the one-to-one
  correspondence above was also obtained in \cite{CrSaSt12} for
  $\Delta_M = TM$. The map $\bbd^{\Delta}$ is called by the authors the
  \textit{Spencer operator relative to $\pi$}.
\end{remark}


\subsection{Infinitesimal ideal systems, distributions on Lie algebroid and subrepresentations.}

Let us recall the definition of an infinitesimal ideal system
\cite{Hawkins08, JoOr13} on a Lie algebroid $(A, [\cdot, \cdot]_A,
\rho_A)$.

\begin{definition}
  An infinitesimal ideal system on $A$ is a triple $(\Delta_M, C,
  \widetilde{\nabla})$, where $C \subset A$ is a subalgebroid,
  $\Delta_M \subset TM$ is a integrable distribution and
  $\widetilde{\nabla}: \Gamma(\Delta_M) \times \Gamma(A/C) \Arrow
  \Gamma(A/C)$ is a flat connection satisfying the following
  properties:
\begin{enumerate}
\item if $\pi(a)$ is parallel, then $[a, \Gamma(C)]_A \subset \Gamma(C)$.
\item if $\pi(a), \pi(b)$ are parallel, then $\pi([a,b]_A)$ is parallel;
\item if $\pi(a)$ is parallel, then $[\rho(a), \Gamma(\Delta_M)]
  \subset \Gamma(\Delta_M)$;
\end{enumerate} 
where $a, b \in \Gamma(A)$ and $\pi: A \Arrow A/C$ is the quotient map.
\end{definition}

Given an infinitesimal ideal system $(\Delta_M, C,
\widetilde{\nabla})$ on $A$, it follows from Theorem \ref{spencer_thm}
that there exists an associated linear distribution $(\Delta; \Delta_M,
A; M)$ on $A$ having core $C$ corresponding to an operator $\bbd:
\Gamma(A) \Arrow \Gamma(\Delta^*_M \otimes (A/C))$ defined as zero on
$C$ and equal to $\widetilde{\nabla}$ on the quotient $A/C$. 
The paper \cite{JoOr13} shows that the properties of an infinitesimal
ideal system are equivalent to $\Delta \Arrow \Delta_M$ and $\Delta
\Arrow A$ being Lie subalgebroids of $TA
\Arrow TM$ and $TA \Arrow A$, respectively. In this section, we prove
this in an alternative manner, using representations up to
homotopy. Along the way, we shall understand necessary and
sufficient conditions on $\bbd$ for $\Delta \Arrow \Delta_M$
to be a Lie subalgebroid of $TA\Arrow TM$ and for  and
$\Delta \Arrow A$ to be a Lie subalgebroid of
$TA \Arrow A$.

Let us start with the definition of $\VB$-subalgebroids.
\begin{definition}
  Let $(D' \Arrow B'; A' \Arrow M)$ be a $\VB$-algebroid. We say that
  a double vector subbundle $(D; A, B; M)$ is a
  \textit{$\VB$-subalgebroid} of $D'$ if $D \Arrow B$ is a Lie
  subalgebroid of $D' \Arrow B'$.
\end{definition}

\begin{proposition}\label{sub_equiv}
  A double vector subbundle $(D; A, B; M)$ is a
  \textit{$\VB$-subalgebroid} of $D'$ if and only if
  \begin{enumerate}
\item $(D \Arrow B; A \Arrow M)$ is a $\VB$-algebroid;
\item the inclusion map $i: (D \Arrow B; A \Arrow M) \hookrightarrow
  (D' \Arrow B'; A' \Arrow M)$ is a $\VB$-algebroid morphism.
\end{enumerate}
\end{proposition}

\begin{proof}
  It is straightforward to see that conditions (1) and (2) imply that
  $D$ is a $\VB$-subalgebroid of $D'$. Conversely, assume that $D
  \Arrow B$ is a Lie subalgebroid of $D' \Arrow B'$. The fact that the
  inclusion $i: D \Arrow D'$ is a bundle morphism over $i_A: A \Arrow
  A'$ implies that the anchor of $D$, $\rho_D= \rho_{D'} \circ i$, is
  a bundle morphism over $\rho_A=\rho_{A'}\circ i_A$. To prove that
  $(D \Arrow B; A \Arrow M)$ is a $\VB$-algebroid, we still have to
  check conditions (i), (ii) and (iii) of Definition
  \ref{VB_def}. These will follow from exactly the same conditions on
  $D'$ if we prove that core (respectively linear) sections of $D$ can
  be extended to core (respectively linear) sections of $D'$. Now,
  \eqref{core_section} implies that
$$
\Gamma_{c}(B, D) = \{\widehat{c}\,|_B \,\, | \,\, c \in \Gamma(C)
\text{ and } \widehat{c} \in \Gamma_{c}(B', D')\}.
$$
Also, for $\mathcal{X}: B \Arrow D$, a linear section of $D$ covering
$a \in \Gamma(A)$, choose any horizontal lift $h': A' \Arrow
\widehat{A'}$ adapted to $D$ (the existence of which is guaranteed by
Proposition \ref{adapted_prop}). For $b \in B$,
$$
\mathcal{X}(b) \dif{B} h(a)(b) = \widetilde{0}^B_{b} \soma{A} \overline{\Phi_{a}(b)},
$$
for some $\Phi \in \Omega^1(A, \Hom(B, C))$. Extend $\Phi$ to $\Phi' \in
\Omega(A', \Hom(B', C'))$. The horizontal lift $h''= h' + \Phi'$ is
still adapted to $D$ and $\mathcal{X}= h''(a)|_B$.
\end{proof}

\begin{definition}
  Let $W \in \R \mathrm{ep}^{2}(A)$ be a 2-term representation of a
  Lie algebroid $A$ and let $(\partial, \nabla, K)$ be its structure
  operators. We say that a (graded) subbundle $V \subset W$ is a
  \textit{subrepresentation} if $V \in \R \mathrm{ep}^2(A)$ and
  $(i_{V_0}, i_{V_1}, 0)$ are the components of a morphism $(A, V)
  \Rightarrow (A,W)$, where $i_{V_0}: V_0 \hookrightarrow W_0$ and
  $i_{V_1}: V_1 \hookrightarrow W_1$ are the inclusions and $0 \in
  \Gamma(A^*\otimes V_1^*\otimes W_0)$.
\end{definition}

\begin{remark}\em
  If $(\partial, \nabla, K)$ are the structure operators for $W \in \R
  \mathrm{ep}^2(A)$, then $V \subset W$ is a subrepresentation if and
  only if
\begin{align}
\label{sub1} & \vspace{-20pt}\partial(V_0) \subset V_1; \\
\label{sub2} & \vspace{-20pt}\nabla_a \text{ preserves } V,\, \forall \,a \in \Gamma(A)\\
\label{sub3} & \vspace{-20pt} K(a_1,a_2)(V_1) \subset V_0, \forall \, a_1, \,a_2 \in A.
\end{align}
In this case, the restrictions $(\partial|_{V_0}, \nabla|_V,
K|_{V_1})$ are the structure operators for the representation up to
homotopy of $A$ on $V$. This follows directly from equations
\eqref{comp1}, \eqref{comp2} and \eqref{comp3}.
\end{remark}

Let us give an example.

\begin{example}\label{sub_double}
  Let $\nabla: \Gamma(TM)\times \Gamma(B) \Arrow \Gamma(B)$ be a
  connection on $B$ and consider the double representation
  $\D_{\nabla}(B) \in \R \mathrm{ep}^2(TM)$ (see Example
  \ref{double}). Given a vector subbundle $C \subset B$, the
  graded vector bundle $C_{[0]}\oplus B_{[1]}$ is a subrepresentation
  if and only if $\nabla$ preserves $C$ and the induced connection on
  the quotient $\widehat{\nabla}: \Gamma(TM) \times \Gamma(B/C) \Arrow
  \Gamma(B/C)$ is flat.
\end{example}

The following result shows how $\VB$-subalgebroids and
subrepresentations are related.

\begin{theorem}\label{subrep}
  Let $W=C'_{[0]}\oplus B'_{[1]} \in \R \mathrm{ep}^{2}(A')$ be the
  representation up to homotopy corresponding to a $\VB$-algebroid
  structure on $(A'\oplus B'\oplus C' \Arrow B'; A' \Arrow M)$. Then
  $(A\oplus B\oplus C \Arrow B; A \Arrow M)$ is a $\VB$-subalgebroid
  if and only if $A \subset A'$ is a subalgebroid and $C_{[0]}\oplus
  B_{[1]}$ is a subrepresentation of $i_A^!W \in \R
  \mathrm{ep}^{2}(A)$, where $i_A: A \hookrightarrow A'$ is the
  inclusion.
\end{theorem}

\begin{proof}
  Assume $A\oplus B\oplus C$ is a $\VB$-subalgebroid of $D'$. By
  Proposition \eqref{sub_equiv}, it follows that $(A\oplus B\oplus C
  \Arrow B; A \Arrow M)$ is a $\VB$-algebroid, so there is a
  corresponding Lie algebroid structure on $A$ and $V=C_{[0]}\oplus
  B_{[1]} \in \Rep^{2}(A)$. As the inclusion $i: A\oplus B\oplus C
  \hookrightarrow A'\oplus B'\oplus C'$ is a $\VB$-morphism, it
  follows from Theorem \ref{main} that the inclusion $i_A: A \Arrow
  A'$ is a Lie algebroid morphism and $(i_C, i_B, 0)$ are the
  components of a morphism $(A, V) \Rightarrow (A', W)$ over the
  inclusion $i_A$.  Conversely, assume that $A \subset A'$ is a
  subalgebroid and that $V$ is a subrepresentation of $i_A^!W$. The
  representation up to homotopy of $A$ on $V$ give $(A\oplus B\oplus C
  \Arrow B; A \Arrow M)$ a $\VB$-algebroid structure and one can use
  Theorem \eqref{main} once again to prove that the inclusion $i:
  A\oplus B\oplus C \Arrow A'\oplus B'\oplus C'$ is a
  $\VB$-morphism. This concludes the proof.
\end{proof}

\begin{corollary}\label{involutive}\cite{CrSaSt12, JoOr13}
  A linear distribution $(\Delta; \Delta_M, B; M)$ on $B$ with core
  bundle $C$ is involutive if and only if $\Delta_M$ is involutive and
  the associated map $\bbd^{\Delta}: \Gamma(B) \Arrow
  \Gamma(\Delta_M^*\otimes (B/C))$ satisfies
\begin{itemize}
 \item[(1)] $\bbd^{\Delta}|_{\Gamma(C)} = 0;$
 \item[(2)] the map induced on the quotient $\Gamma(B/C) \Arrow
   \Gamma(\Delta_M^{*}\otimes (B/C))$ is a flat $\Delta_M$-connection
   on $B/C$.
\end{itemize}
\end{corollary}

\begin{proof}
  $\Delta$ is involutive if and only if $(\Delta; \Delta_M; A; M)$ is
  a $\VB$-subalgebroid of the double $(TA \Arrow A; TM \Arrow
  M)$. Choose any connection $\nabla: \Gamma(TM) \times \Gamma(A)
  \Arrow \Gamma(A)$ with
$$
\pi \circ \nabla_x = \bbd^{\Delta}_x, \, \forall \, x \in \Gamma(\Delta_M)
$$
and consider the double representation $\D_{\nabla}(B) \in \R
\mathrm{ep}^{2}(TM)$. It is the representation up to homotopy of $TM$
associated to the $\VB$-algebroid $(TB \Arrow B; TM \Arrow M)$
decomposed by the choice of $\nabla$. As $\nabla$ is adapted to
$\Delta$ (see Theorem \ref{spencer_thm}), it follows from Theorem
\ref{subrep} that $\Delta$ is involutive if and only if $\Delta_M$ is
involutive and $C_{[0]}\oplus B_{[1]}$ is a subrepresentation of
$i_{\Delta_M}^{\,!} \D_{\nabla} \in \R \mathrm{ep}^2(\Delta_M)$, where
$i_{\Delta_m}: \Delta_M \hookrightarrow TM$ is the inclusion. The
result now follows from Example \ref{sub_double}.
\end{proof}

\subsubsection*{Compatibility with the Lie algebroid structure.}

Let $\nabla: \Gamma(TM) \times \Gamma(A) \Arrow \Gamma(A)$ be a
connection on the Lie algebroid $A$ and consider the adjoint
representation $\mathrm{ad}_{\nabla}$. The graded subbundle $C_{[0]}
\oplus \Delta_{M, \, [1]} \subset A_{[0]}\oplus TM_{[1]}$ is a
subrepresentation of $\mathrm{ad}_{\nabla}$ if and only if $\rho_A(C)
\subset \Delta_M$ and
\begin{align}
  \label{adj:1} \, [a,c] + \nabla_{\rho(c)} \, a \in \Gamma(C);\\
\label{adj:2}[\rho_A(a), x] + \rho_A(\nabla_x a) \in \Gamma(\Delta_M);\\
\label{adj:3} R_{bas}(a,b)(\Delta_M) \subset C;
\end{align}
for $a, b \in \Gamma(A)$, $c \in \Gamma(C)$ and $x \in
\Gamma(\Delta_M)$. In the case $\nabla$ is a connection adapted to a
linear distribution $(\Delta; A, \Delta_M; M)$, equations
\eqref{adj:1}, \eqref{adj:2} and \eqref{adj:3} can be reinterpreted in
terms of the $\bbd^{\Delta}$ \eqref{mu_delta} to give conditions for
$\Delta \Arrow \Delta_M$ to be a Lie subalgebroid of $TA \Arrow
TM$. In the following, we shall need the quotient maps $\pi: A \Arrow
A/C$ and $\pi_{TM}: TM \Arrow TM/\Delta_M$.

\begin{theorem}\label{IM_prop}
  Let $(\Delta; \Delta_M, A; M)$ be a linear distribution on $A$ with
  core $C$ and choose any connection $\nabla: \Gamma(TM)\times
  \Gamma(A) \Arrow \Gamma(A)$ adapted to $\Delta$. Then
  $\Delta \Arrow \Delta_M$ is a Lie subalgebroid of $TA \Arrow TM$ if
  and only if $\rho_A(C) \subset \Delta_M$,
\begin{align}
\label{adj:d1}&  \bbd_{\rho_A(c)}(a) = -\pi([a, c])\vspace{5pt};\\
\label{adj:d2}& \widetilde{\rho_A}(\bbd_x(a)) = -\pi_{TM}([\rho_A(a),x])\vspace{5pt}\\
\label{adj:d3}&\bbd_x([a,b]_A) = \widehat{\nabla}^{\rm bas}_a
\bbd_x(b) - \widehat{\nabla}^{\rm bas}_b \bbd_x(a) +
\pi(\nabla_{[\rho_A(b), x]} a - \nabla_{[\rho_A(a), x]} b)
\end{align}
where $\widetilde{\rho_A}: A/C \Arrow TM/\Delta_M$ is the quotient map
(i.e.~$\widetilde{\rho_A}\circ \pi = \pi_{TM} \circ \rho_A$) and where
$\widehat{\nabla}^{\rm bas}$ is the $A$-connection on the quotient
$A/C$ given by
$$
\widehat{\nabla}^{\rm bas}_a \pi(b) = \pi([a,b] + \nabla_{\rho(b)}a) = \pi(\nabla_a^{\rm bas} b)\\
$$
for $a, b \in \Gamma(A)$, $c \in \Gamma(C)$ and $x \in \Gamma(\Delta_M)$.
\end{theorem}

\begin{proof}
  As $\nabla$ is adapted to $\Delta$, Theorem \ref{subrep} assures
  that $\Delta$ is compatible with the Lie algebroid structure if and
  only if $V=C_{[0]}\oplus \Delta_M\vphantom{}_{[1]}$ is a
  subrepresentation of $\operatorname{ad}_{\nabla} \in \R
  \mathrm{ep}^2(A)$.  So, one is left to prove that \eqref{adj:1},
  \eqref{adj:2} and \eqref{adj:3} correspond to \eqref{adj:d1},
  \eqref{adj:d2} and \eqref{adj:d3}. Now, by applying the quotient
  maps $\pi$ and $\pi_{TM}$ to \eqref{adj:1} and \eqref{adj:2}, one
  has
$$
\left\{
\begin{array}{l}
  \left[a, c\right] + \nabla_{\rho_A(c)} \, a \in \Gamma(C)\vspace{5pt}\\
  \left[\rho_A(a), x\right] + \rho_A(\nabla_x a) \in \Gamma(\Delta_M)\\
\end{array}
\right.
\Longleftrightarrow
\left\{
\begin{array}{l}
 \pi(\nabla_{\rho_A(c)}a) = - \pi(\left[a, c\right])\vspace{5pt}\\
 \pi_{TM}(\rho_A(\nabla_x a)) = - \pi_{TM}(\left[\rho_A(a), x\right])  \\
\end{array}
\right.
$$
for every $c \in \Gamma(C)$, $a \in \Gamma(A)$ and $x \in
\Gamma(\Delta_M)$. The result now follows from the fact
that, for any adapted connection $\nabla$, $\bbd_x = \pi\circ
\nabla_x$, $\forall \, x \in \Gamma(\Delta_M)$ and $\pi_{TM}\circ
\rho_A= \widetilde{\rho}_{A} \circ \pi$.

At last, using the explicit expression for $R^{\rm bas}$, one has that
\eqref{adj:3} holds if and only if
\begin{equation}\label{eq_curv}
\pi(\nabla_x [a, b] - [\nabla_x a, b]-[a, \nabla_x b] - \nabla_{\nabla_b^{\rm bas} x } \,a 
+ \nabla_{\nabla_a^{\rm bas} x }\, b) = 0,
\end{equation}
for $a, b \in \Gamma(A)$ and $ x\in \Gamma(\Delta_M)$.
Now, note that
$$
- [\nabla_x a, b] + \nabla_{\nabla^{\rm bas}_a x} b = [b, \nabla_x a] +
\nabla_{\rho_{A}(\nabla_x a)} b + \nabla_{[\rho_A(a), x]} b =
\nabla^{\rm bas}_b \nabla_x a + \nabla_{[\rho_A(a), x]} b
$$
and similarly
$$
[a, \nabla_x b] + \nabla_{\nabla_b^{\rm bas} x } a= \nabla^{\rm bas}_a \nabla_x b + \nabla_{[\rho_A(b), x]} a.
$$
So, by definition of $\widehat{\nabla}^{\rm bas}$, one has that
\eqref{eq_curv} is equivalent to
$$
\bbd_x([a,b]) + \widehat{\nabla}^{\rm bas}_b \bbd_x(a) -
\widehat{\nabla}^{\rm bas}_a \bbd_x(b) + \pi(\nabla_{[\rho(a), x]} b -
\nabla_{[\rho_A(b), x]} a) = 0,
$$
as required.
\end{proof}

\begin{remark}\em
In the particular case where $\Delta_M = TM$, one can get rid of the
choice of an adapted connection. Indeed, note that
$\widehat{\nabla}^{\rm bas}$ can be alternatively given by
$$
\widehat{\nabla}^{\rm bas}_a b = \pi([a,b]) + \bbd_{\rho(b)}(a)
$$
and \eqref{adj:d3} becomes
$$
\bbd_x([a,b]_A) = \widehat{\nabla}^{\rm bas}_a \bbd_x(b) -
\widehat{\nabla}^{\rm bas}_b \bbd_x(a) + \bbd_{[\rho_A(b), x]} (a) -
\bbd_{[\rho_A(a), x]} (b).
$$
In this form, Theorem \ref{IM_prop} gives the infinitesimal
counterpart of a result from \cite{CrSaSt12} characterizing (wide)
multiplicative distributions in terms of Spencer operators relative to $\pi$.
\end{remark}

Our last result explains how infinitesimal ideal systems and
representations up to homotopy are related.

\begin{theorem}\label{ideals}
  Let $A$ be a Lie algebroid over $M$. A triple $(\Delta_M, C,
  \widetilde{\nabla})$ is an infinitesimal ideal system if and only if
\begin{enumerate}
\item $\Delta_M \subset TM$ is integrable;
\item $C_{[0]}\oplus \Delta_M\vphantom{}_{[1]}$ is a subrepresentation
  of $\mathrm{ad}_{\nabla}(A)$ and
\item $C_{[0]}\oplus A_{[1]}$ is a subrepresentation of $i_{\Delta_M}^!\D_{\nabla}(A)$,
\end{enumerate}
$i_{\Delta_M}: \Delta_M \hookrightarrow TM$ is the inclusion and
$\nabla: \Gamma(TM) \times \Gamma(A) \Arrow \Gamma(A)$ is any
connection preserving $C$ and inducing $\widetilde{\nabla}$ on the
quotient.
\end{theorem}

\begin{proof}
Define $\bbd: \Gamma(A) \Arrow \Gamma(\Delta_M^*\otimes (A/C))$ by 
$$
\bbd_x(a) = \widetilde{\nabla}_x \pi(a), \, a \in \Gamma(A), \, x \in \Gamma(\Delta_M).
$$
By Theorem \eqref{spencer_thm}, there exists a linear distribution
$(\Delta; \Delta_M, A; M)$ with core $C$ such that $\bbd =
\bbd^{\Delta}$ \eqref{mu_delta}. We already know (see \cite{JoOr13})
that $(\Delta_M, C, \widetilde{\nabla})$ is an infinitesimal ideal
system if and only $\Delta \Arrow A$ and $\Delta \Arrow \Delta_M$ are
simultaneously Lie subalgebroids of $TA \Arrow A$ and $TA \Arrow TM$,
respectively. The result now follows directly from Theorem
\eqref{subrep}.
\end{proof}


\appendix

\section{Dualization of $\VB$-algebroids and representations up to homotopy.}\label{dual}
Let $(D; A, B; M)$ be a double vector bundle and consider its horizontal
\eqref{hor_dual} and vertical \eqref{ver_dual} duals. The vector
bundles $p^{\ver}_{\vphantom{C}_{C^*}}: D^*_A\Arrow C^*$ and
$p^{\hor}_{\vphantom{C}_{C^*}}:D^*_B\Arrow C^*$ are dual to each
other 
via the nondegenerate  pairing $\Vert \cdot,
\cdot \Vert:D^*_A\times_{C^*}D^*_B\Arrow \R$ given by
\begin{equation}\label{pairingverhorduals}
\Vert \Theta, \Psi \Vert :=  \langle \Psi, d \rangle _B - \langle \Theta, d \rangle _A
\end{equation}
where $d\in D$ is any element with $q^D_A(d)=p_A(\Theta)$
and $q^D_B(d)=p_B(\Psi)$ \cite{Mackenzie05}. The pairings on the right-hand side of
\eqref{pairingverhorduals} are defined with respect to the fibers over
$B$ and over $A$, respectively. Henceforth, we identify the dual of
$D^*_B\Arrow C^*$ with $D^*_A\Arrow C^*$ via the pairing in
\eqref{pairingverhorduals}. Given a section $\Theta: C^* \Arrow
D_A^*$, we denote by $\ell_{\Theta}^{C^*} \in C^{\infty}(D_B^*)$ the
function which is linear with respect to the vector bundle structure
$D^*_B \Arrow C^*$ and given
by
$$
\ell_{\Theta}^{C^*}(\Psi) = \Vert \Theta(p^{\hor}_{\vphantom{C}_{C^*}}(\Psi)), \Psi \Vert.
$$
In particular, for the core section $\widehat{\psi}\in\Gamma(C^*, D_A^*)$
associated to some $\psi \in \Gamma(B^*)$, one gets by choosing $d=0^D_{p_B(\Psi)}$
in \eqref{pairingverhorduals}:
\begin{equation}\label{eq1}
\ell^{C^*}_{\widehat{\psi}}=-\ell_{\psi}\circ p_B.
\end{equation}
Also, for $T \in \Gamma(B^*\otimes C)$,
\begin{equation}\label{dual_corresp}
\ell^{C^*}_{\widehat{T^*}} = -\ell_{\widehat{T}},
\end{equation}
where $\widehat{T} \in \Gamma_{\ell}(B, D)$ and $\widehat{T^*} \in
\Gamma_{\ell}(C^*, D_A^*)$ are the linear sections \eqref{core_morf}
corresponding to $T$ and its dual $T^*\in \Gamma(C^*\otimes B)$,
respectively. 

We shall need one more formula (which follows directly from
 \eqref{C_proj}) for $c \in \Gamma(C)$ and the corresponding core
 section $\widehat{c}\in\Gamma(B, D)$, namely
\begin{equation}\label{eq2}
\ell_{\widehat{c}} = \ell_{c}\circ p^{\hor}_{C^*}.
\end{equation}

\bigskip
Assume now that $(D\Arrow B;A \Arrow M)$ is a
$\mathcal{VB}$-algebroid. Then $(D_A^* \Arrow C^*; A \Arrow M)$ has a
natural $\VB$-algebroid structure.  The Lie algebroid structure on
$D_A^* \Arrow C^*$ is obtained by noticing that the linear Poisson
structure on $D_B^* \Arrow B$ associated to the Lie algebroid $D
\Arrow B$ is also linear with respect to the vector bundle structure
$D^*_B \Arrow C^*$ \footnote{Let $q_E:E\to M$ be a vector bundle. A
  Poisson bracket $\{\cdot \,,\cdot\}$ on $C^\infty(E)$ is linear if
  for all $\xi,\xi'\in\Gamma(E^*)$ and $f,f'\in C^\infty(M)$, the
  bracket $\{\ell_\xi,\ell_{\xi'}\}$ is again linear, the bracket
  $\{q_E^*f,q_E^*f'\}$ vanishes and $\{\ell_\xi,q_E^*f\}$ is the
  pullback under $q_E$ of a function on $M$.  This defines a Lie
  algebroid $(E^*\to M,\rho, [\cdot\,,\cdot])$ by setting
  $\ell_{[\xi,\xi']}:=\{\ell_\xi,\ell_{\xi'}\}$ and
  $q_E^*(\Lie_{\rho(\xi)}f):=\{\ell_\xi,q_E^*f\}$.  Conversely, a Lie
  algebroid structure on $E^*$ defines a linear Poisson structure on
  $E$.  }. In particular, besides the usual formulas
$$
\ell_{[\chi_1, \chi_2]_{D}} = \{\ell_{\chi_1},
\ell_{\chi_2}\}_{D_B^*}, \,\, \Lie_{\rho_{D}(\chi)}(f) \circ p_B =
\{\ell_{\chi}, f \circ p_B\}_{D_{B}^*},
$$
defining the Lie bracket $[\cdot,\cdot]_{D}$ and the anchor $\rho_D$
on $D \Arrow B$, for $\chi, \chi_1, \chi_2 \in \Gamma(B, D)$ and $f
\in C^{\infty}(B)$, we have
$$
\ell^{C^*}_{[\Theta_1, \Theta_2]_{D^*_{A}}} = \{\ell^{C^*}_{\Theta_1},
\ell^{C^*}_{\Theta_2}\}_{D_B^*}, \,\, \Lie_{\rho_{D_A^*}(\Theta)}(g)
\circ p_{C^*}^{\hor} = \{\ell_{\Theta}^{C^*}, g \circ
p^{\hor}_{C^*}\}_{D_{B}^*},
$$
defining the Lie bracket $[\cdot, \cdot]_{D_A^*}$ and the anchor
$\rho_{D_A^*}$ on $D_A^* \Arrow C^*$, for $\Theta, \Theta_1, \Theta_2
\in \Gamma(C^*, D_A^*)$ and $g \in C^{\infty}(C^*)$. We refer to
\cite{Mackenzie05} (see also \cite{GrMe10a}) for more details.

Our aim here is to prove that the representations up to homotopy
associated to $D \Arrow B$ and $D_A^* \Arrow C^*$ are dual to each
other. Let us first consider how horizontal lifts for $(D; A, B; M)$
and $(D_A^*; A, C^*; M)$ are related. Let $h:\Gamma(A) \Arrow
\Gamma_{\ell}(B, D)$ be a horizontal lift for $D$. There exists a
corresponding horizontal lift $h^{\top}:\Gamma(A)\rmap
\Gamma_{\ell}(C^*, D_A^*)$ given as follows: take the decomposition
$\sigma_h \in \dec{D}$ associated to $h$ by \eqref{decomp} and
consider the inverse of its dual over $A$, $(\sigma_h)_A^{*^{-1}} \in
\dec{D_A^*}$. Set $h^{\top}$ to be the horizontal lift corresponding
to $(\sigma_h)_A^{*^{-1}}$ via \eqref{hor_decomp}.  It is
straightforward to check that
\begin{equation}\label{horliftcorrespondence}
\ell_{h(a)}=\ell^{C^*}_{h^{\top}(a)} \in C^{\infty}(D_B^*)
\end{equation}
for every $a\in \Gamma(A)$.

Let $(\partial, \nabla, K)$ be the structure operators of the
representation up to homotopy $C_{[0]}\oplus B_{[1]} \in \R
\mathrm{ep}^2(A)$ associated to $(D, h)$ and $(\partial^{\ver},
\nabla^{\ver}, K^{\ver})$ be the structure operators of the
representation up to homotopy $B^*_{[0]}\oplus C^*_{[1]} \in \R
\mathrm{ep}^2(A)$ associated to $(D^*_A,h^{\top})$. The next result
relates the two representations.

\begin{proposition}\label{dual_prop}
  The structure operators $(\partial^{\ver}, \nabla^{\ver}, K^{\ver})$
  coincide with the structure operators \eqref{dual_struct} of the
  representation $(C_{[0]}\oplus B_{[1]})^{\top} \in \R
  \mathrm{ep}^2(A)$ dual to $(\partial, \nabla, K)$.
\end{proposition}

\begin{proof}
Let $c \in \Gamma(C)$ and $\psi \in
  \Gamma(B^*)$. By \eqref{D_anchor},
 we have 
$$
\begin{array}{rl}
  \langle \partial^{\ver}(\psi), c\rangle \circ
  q_{C^*}\circ p^{\hor}_{C^*} & = \Lie_{\rho_{D_A^*}(\widehat{\psi})}(\ell_c) \circ p_{C^*}^{\hor}  = \{\ell^{C^*}_{\widehat{\psi}}, \ell_c\circ p^{\hor}_{C^*}\}_{D^*_B}\vspace{3pt}\\
  & = -\{\ell_{\psi}\circ p_B, \,\ell_{\widehat{c}}\}_{D^*_B} = \Lie_{\rho_D(\widehat{c})}(\ell_{\psi}) \circ p_B\vspace{5pt}\\
  & = \langle \psi, \partial(c) \rangle \circ q_B \circ p_B.
\end{array}
$$
Since $q_C^* \circ p_{C^*}^{\hor} = q_B \circ p_B$, the equality
$\partial^{\ver}=\partial^*$ follows.

Let us prove the relation between the $A$-connections. By
\eqref{D_anchor} and \eqref{c} together with \eqref{eq1} and
\eqref{horliftcorrespondence}, one has that
$$
\begin{array}{rl}
  \ell_{\nabla^{\ver}_a
    \psi} \circ p_B & =  - \ell^{C^*}_{\widehat{\nabla^{\ver}_a\psi}}  = - \ell^{C^*}_{[h^{\top}(a),\widehat{\psi}]_{D^*_A}} = - \{\ell^{C^*}_{h^{\top}(a)},\ell^{C^*}_{\hat{\psi}}\}_{D^*_B} \vspace{5pt}\\
  & =  \{\ell_{h(a)}, \ell_{\psi}\circ p_B\}_{D^*_B} = \Lie_{\rho_D(h(a))}(\ell_{\psi}) \circ p_B = \ell_{\nabla_a^*\psi} \circ p_B.
\end{array}
$$
Similarly,
$$
\begin{array}{rl}
  \ell_{\nabla^{\ver \vphantom{}^*}_a
    c} \circ p_{C^*}^{\hor} & = \Lie_{\rho_{D_A^*}(h^{\top}(a))} (\ell_{c})\circ p_{C^*}^{\hor}  = \{\ell^{C^*}_{h^{\top}(a)}, \ell_c \circ p_{C^*}^{\hor} \}_{D_B^*} =  \{\ell_{h(a)}, \ell_{\widehat{c}} \}_{D^*_B} \vspace{5pt}\\
  & =  \ell_{[h(a), \widehat{c}]_D}= \ell_{\widehat{\nabla_a c}} = \ell_{\nabla_a c} \circ p_{C^*}^{\hor}.
\end{array}
$$
Hence, we have verified the equality $\nabla^{\ver} = \nabla^*$.

It remains to compare the curvatures. For that, choose
$a_1,a_2\in\Gamma(A)$ and let $\widehat{K} \in \Gamma_{\ell}(B, D)$
and $\widehat{K^{\ver}} \in \Gamma_{\ell}(C^*, D_A^*)$ be the linear
sections \eqref{core_morf} corresponding to $K(a_1,a_2) \in
\Gamma(B^*\otimes C)$ and $K^{\ver}(a_1,a_2) \in \Gamma(C^* \otimes
B)$ respectively. First note that, by \eqref{horliftcorrespondence},
$$
\ell^{C^*}_{[h^{\top}(a_1),\, h^{\top}(a_2)]_{D_A^*}} =
\left\{\ell^{C^*}_{h^{\top}(a_1)},
  \ell^{C^*}_{h^{\top}(a_2)}\right\}_{D_B^*} = \left\{\ell_{h(a_1)},
  \ell_{h(a_2)}\right\}_{D_B^*} = \ell_{[h(a_1), h(a_2)]_D}
$$
Therefore, by \eqref{curvature} and \eqref{dual_corresp}, we find
$$
\ell^{C^*}_{\widehat{K^{\ver}}} = \ell^{C^*}_{[h^{\top}(a_1),
  h^{\top}(a_2)]_{D_A^*} -h^{\top}([a_1,a_2]_A) } = \ell_{ [h(a_1),
  h(a_2)]_D - h([a_1,a_2]_A)} = \ell_{\widehat{K}} = -
\ell^{C^*}_{\widehat{K^*}}
$$
This proves that $K^{\ver} =  -K^*$.
\end{proof}

\def\cprime{$'$} \def\polhk#1{\setbox0=\hbox{#1}{\ooalign{\hidewidth
  \lower1.5ex\hbox{`}\hidewidth\crcr\unhbox0}}}


\begin{thebibliography}{10}

\bibitem{ArCr12}
C.~Arias~Abad and M.~Crainic.
\newblock Representations up to homotopy of {L}ie algebroids.
\newblock {\em J. Reine Angew. Math.}, 663:91--126, 2012.

\bibitem{BuCa12}
H.~Bursztyn and A.~Cabrera.
\newblock Multiplicative forms at the infinitesimal level.
\newblock {\em Math. Ann.}, 353(3):663--705, 2012.

\bibitem{BuCaOr09}
H.~Bursztyn, A.~Cabrera, and C.~Ortiz.
\newblock Linear and multiplicative 2-forms.
\newblock {\em Lett. Math. Phys.}, 90(1-3):59--83, 2009.

\bibitem{BuCrWeZh04}
H.~Bursztyn, M.~Crainic, A.~Weinstein, and C.~Zhu.
\newblock Integration of twisted {D}irac brackets.
\newblock {\em Duke Math. J.}, 123(3):549--607, 2004.

\bibitem{CrSaSt12}
M.~Crainic, M.A. Salazar, and I.~Struchiner.
\newblock Linearization of multiplicative forms.
\newblock {\em arXiv:1210.2277}, 2012.

\bibitem{EvLuWe99}
S.~Evens, J-H.~ Lu, and A.~ Weinstein.
\newblock Transverse measures, the modular class and a cohomology pairing for Lie algebroids.
\newblock {\em The Quarterly Journal of Mathematics} 50: 417--436, 1999.

\bibitem{CrFer05}
M.~Crainic and R.L.~Fernandes.
\newblock Secondary characteristic classes of Lie algebroids.
\newblock {\em Quantum field theory and noncommutative geometry}, volume 662 of {\em Lectures Notes in  Physics}, pages 157-176. Springer, Berlin, 2005.

\bibitem{GrMe10a}
A.~Gracia-Saz and R.~A. Mehta.
\newblock Lie algebroid structures on double vector bundles and representation
  theory of {L}ie algebroids.
\newblock {\em Adv. Math.}, 223(4):1236--1275, 2010.

\bibitem{Hawkins08}
E.~Hawkins.
\newblock A groupoid approach to quantization.
\newblock {\em J. Symplectic Geom.}, 6(1):61--125, 2008.

\bibitem{Jotz11}
M.~Jotz.
\newblock Dirac {L}ie groups, {D}irac homogeneous spaces and the theorem of
  {D}rinfeld.
\newblock {\em Indiana Univ. Math. J.}, 60:319--366, 2011.

\bibitem{JoOr13}
M.~Jotz and C.~Ortiz.
\newblock Foliated groupoids and infinitesimal ideal systems.
\newblock {\em Preprint, new version of arXiv:1109.4515}, 2013.

\bibitem{Mackenzie05}
K.C.H. Mackenzie.
\newblock {\em General {T}heory of {L}ie {G}roupoids and {L}ie {A}lgebroids},
  volume 213 of {\em London Mathematical Society Lecture Note Series}.
\newblock Cambridge University Press, Cambridge, 2005.

\bibitem{Mackenzie11}
K.C.H. Mackenzie.
\newblock Ehresmann doubles and {D}rinfel'd doubles for {L}ie algebroids and
  {L}ie bialgebroids.
\newblock {\em J. Reine Angew. Math.}, 658:193--245, 2011.

\bibitem{MaXu94}
K.C.H. Mackenzie and P.~Xu.
\newblock {Lie bialgebroids and Poisson groupoids.}
\newblock {\em Duke Math. J.}, 73(2):415--452, 1994.

\bibitem{Ortiz08}
C.~Ortiz.
\newblock {Multiplicative Dirac structures on Lie groups.}
\newblock {\em C. R., Math., Acad. Sci. Paris}, 346(23-24):1279--1282, 2008.

\end{thebibliography}
\end{document}